\documentclass[11pt]{amsart}
\usepackage{fullpage}
\usepackage[colorlinks=true,citecolor=black!60!green,linkcolor=black!60!red,filecolor=black!60!cyan,urlcolor=black!60!magenta]{hyperref}
\usepackage{amssymb,amsmath,amsthm,mathtools,extpfeil}
\usepackage[all,cmtip]{xy}
\usepackage{algorithmic}
\usepackage{xcolor}
\usepackage[shortlabels]{enumitem}
\newtheorem{thm}[equation]{Theorem}
\newtheorem{thmA}{Theorem}

\newtheorem*{thm*}{Theorem}
\newtheorem{lem}[equation]{Lemma}
\newtheorem{prop}[equation]{Proposition}
\newtheorem{cor}[equation]{Corollary}
\newtheorem{conj}[equation]{Conjecture}
\theoremstyle{definition}
\newtheorem*{defn}{Definition}
\newtheorem{rmk}[equation]{Remark}
\newtheorem{ex}[equation]{Example}
\newtheorem*{exs}{Examples}
\numberwithin{equation}{section}

\DeclareMathOperator{\id}{id}

\DeclareMathOperator{\Gr}{Gr}
\DeclareMathOperator{\Hom}{Hom}

\DeclareMathOperator{\Tor}{Tor}
\DeclareMathOperator{\Ext}{Ext}

\DeclareMathOperator{\act}{act}
\DeclareMathOperator{\Hact}{Hact}
\DeclareMathOperator{\add}{add}
\DeclareMathOperator{\rel}{rel}
\DeclareMathOperator{\proj}{proj}
\DeclareMathOperator{\cd}{cd}
\newcommand{\incl}{i}
\newcommand{\HtoY}{u}
\newcommand{\hatHtoY}{\hat u}
\newcommand{\YtoH}{v}
\newcommand{\hatYtoH}{\hat v}

\newcommand{\zerosec}{o}
\newcommand{\ph}{\varphi}

\newcommand{\eqrefb}[1]{(\ref*{#1})}
\newcommand{\ord}{\kappa}

\begin{document}
\title{Rational homotopy type and computability}
\author{Fedor Manin}
\address{Department of Mathematics, UCSB, Santa Barbara, California, USA}
\email{manin@math.ucsb.edu}
\begin{abstract}
  Given a simplicial pair $(X,A)$, a simplicial complex $Y$, and a map
  $f:A \to Y$, does $f$ have an extension to $X$?  We show that for a fixed
  $Y$, this question is algorithmically decidable for all $X$, $A$, and $f$ if
  $Y$ has the rational homotopy type of an H-space.  As a corollary, many
  questions related to bundle structures over a finite complex are likely
  decidable.  Conversely, for all other $Y$, the question is at least as hard as
  certain special cases of Hilbert's tenth problem which are known or suspected
  to be undecidable.
\end{abstract}
\dedicatory{In memory of Edgar H.\ Brown, 1926--2021}
\maketitle

\section{Introduction}

When can the set of homotopy classes of maps between spaces $X$ and $Y$ be
computed?  That is, when can this (possibly infinite) set be furnished with a
finitely describable and computable structure?  It is reasonable to restrict the
question to the setting of finite complexes: otherwise one risks encountering
spaces that themselves take an infinite amount of information to describe.
Moreover, the question of whether this set has more than one element is
undecidable for $X=S^1$, as shown by Novikov as early as 1955\footnote{This is
  the triviality problem for group presentations, translated into topological
  language.  This work was extended by Adian and others to show that many other
  properties of nonabelian group presentations are likewise undecidable.}.
Therefore it is also reasonable to require the fundamental group not to play a
role; in the present work, $Y$ is always assumed to be simply
connected.\footnote{The results can plausibly be extended to nilpotent spaces.}

We answer this question with the following choice of quantifiers: for what $Y$
and $n$ can the set of homotopy classes $[X,Y]$ be computed for \emph{every}
$n$-dimensional $X$?  Significant partial results in this direction were obtained
by E.~H.~Brown \cite{Brown} and much more recently by \v{C}adek et
al.~\cite{CKMSVW,CKMVW2,CKMVW} and Vok\v r\'inek \cite{Vok}.  The goal of the
present work is to push their program closer to its logical limit.



To state the precise result, we need to sketch the notion of an H-space, which is
defined precisely in \S\ref{S:H}.  Essentially, an H-space is a space equipped
with a binary operation which can be more or less ``group-like''; if it has good
enough properties, this allows us to equip sets of mapping classes to the H-space
with a group structure.

The \emph{cohomological dimension} $\cd(X,A)$ of a simplicial or CW pair $(X,A)$
is the least integer $d$ such that for all $n>d$ and every coefficient group
$\pi$, $H^n(X,A;\pi)=0$.
\begin{thmA} \label{main}
  Let $Y$ be a simply connected simplicial complex of finite type and $d \geq 2$,
  and suppose
  \begin{equation}
    \text{\parbox{.85\textwidth}{
    $Y$ has the rational homotopy type of an H-space through dimension $d$.
    That is, there is a map from $Y$ to an H-space (or, equivalently, to a
    product of Eilenberg--MacLane spaces) which induces isomorphisms on
    $\pi_n \otimes \mathbb{Q}$ for $n \leq d$.}} \tag{$*$} \label{isH}
  \end{equation}
  Then for any simplicial pair $(X,A)$ of cohomological dimension $d+1$ and
  simplicial map $f:A \to Y$, the existence of a continuous extension of $f$ to
  $X$ is decidable.
  
  Moreover, there is an algorithm which, given a simply connected
  simplicial complex $Y$, a simplicial pair $(X,A)$ of finite complexes of
  cohomological dimension $d$ and a simplicial map $f:A \to Y$,
  \begin{enumerate}
  \item Determines whether \eqrefb{isH} is satisfied;
  \item If it is, outputs the set of homotopy classes rel $A$ of extensions
    $[X,Y]^f$ in the format of a (perhaps empty) set on which a finitely
    generated abelian group acts virtually freely and faithfully (that is, with a
    finite number of orbits each of which has finite stabilizer).
  \end{enumerate}
\end{thmA}
We give a few remarks about the statement.  First of all, it is undecidable
whether $Y$ is simply connected; therefore, when given a non-simply connected
input, the algorithm cannot detect this and returns nonsense, like previous
algorithms of this type.

Secondly, we provide evidence for the conjectural converse to the first part of
Theorem \ref{main}: that if \eqrefb{isH} is not true, then the extension problem
for pairs of cohomological dimension $d+1$ is undecidable.  We prove this in a
range of special cases, but the general case appears to be connected to deep
unsolved problems in number theory.  The best that can be said is that if the
converse is false, it must be due to a strange number-theoretic coincidence.

Finally, the difference between $d+1$ in the first part of the theorem and $d$
in the second is significant: if $\cd(X,A)=d+1$, then we can decide whether
$[X,Y]^f$ is nonempty, but our method of describing this set breaks down.  For
example, a homotopy class of maps $S^1 \times S^2 \to S^2$ is determined by two
numbers: the degree $d$ of the map on the $S^2$ factor and the (relative) Hopf
invariant $h$ on the $3$-cell.  However, $h$ is well-defined only up to multiples
of $2d$, and so in a natural sense
\[[S^1 \times S^2,S^2] \cong \bigsqcup_{r \in \mathbb{Z}} \mathbb{Z}/2r\mathbb{Z}.\]
This structure does not fit into the framework we construct in this paper for
describing $[X,Y]^f$.  Other similar examples are described in \cite[\S3]{MW},
and it would be interesting to give a general, perhaps computable, description
for $[X,Y]^f$ (or even just $[X,S^{2n}]$) in this ``critical'' dimension.

\subsection{Examples}
The new computability result encompasses several previous results, as well as new
important corollaries.  Here are some examples of spaces which satisfy condition
\eqrefb{isH} of Theorem \ref{main}:
\begin{enumerate}[(a)]
\item Any simply connected space with finite homology groups (or, equivalently,
  finite homotopy groups) in every dimension is rationally equivalent to a point,
  which is an H-space.  The computability of $[X,Y]$ when $Y$ is of this form was
  already established by Brown \cite{Brown}.
\item Any $d$-connected space is rationally an H-space through dimension $n=2d$.
  Thus we recover the result of \v{C}adek et al.~\cite{CKMVW2} that $[X,Y]^f$ is
  computable whenever $X$ is $2d$-dimensional and $Y$ is $d$-connected.  This
  implies that many ``stable'' homotopical objects are computable.  One example
  is the group of oriented cobordism classes of $n$-manifolds, which is
  isomorphic to the set of maps from $S^n$ to the Thom space of the tautological
  bundle over $\Gr_n(\mathbb{R}^{2n+1})$.
\item The sphere $S^n$ for $n$ odd is rationally equivalent to the
  Eilenberg--MacLane space $K(\mathbb{Z},n)$.  Therefore $[X,S^n]^f$ is
  computable for any finite simplicial pair $(X,A)$ and map $f:A \to S^n$; this
  is the main result of Vok\v r\'inek's paper \cite{Vok}.
\item Any Lie group or simplicial group $Y$ is an H-space, so if $Y$ is simply
  connected then $[X,Y]^f$ is computable for any $X$, $A$, and $f$.
\item Classifying spaces of connected Lie groups also have the rational homotopy
  type of an H-space \cite[Prop.~15.15]{FHT}.  Therefore we have (somewhat
  aspirationally):
  \begin{cor}
    Let $G$ be a connected Lie group, and suppose that the classifying space $BG$
    has a computable representation.  Then:
    \begin{enumerate}[(i)]
    \item Let $X$ be a finite CW complex.  Then the set of isomorphism classes of
      principal $G$-bundles over $X$ is computable.
    \item Let $(X,A)$ be a finite CW pair.  Then it is decidable whether a given
      principal $G$-bundle over $A$ extends over $X$.
    \end{enumerate}
  \end{cor}
  \noindent
  In particular, given a representation $G \to GL_n(\mathbb{R})$, we should be
  able to understand the set of vector bundles with a $G$-structure.  This
  includes real oriented, complex, and symplectic bundles, as well as spin and
  metaplectic structures on bundles.  However, doing this in practice requires
  paying attention to computational models of Lie groups, Grassmannians, bundles,
  and so forth.
\item More generally, some classifying spaces of topological monoids have the
  rational homotopy type of an H-space.  This includes the classifying space
  $BG_n=\operatorname{BAut}(S^n)$ for $S^n$-fibrations (see
  \cite[Appendix 1]{Milnor} and \cite{Smith}); therefore, the set of fibrations
  $S^n \to E \to X$ over a finite complex $X$ up to fiberwise homotopy
  equivalence is computable.
\end{enumerate}
Conversely, most sufficiently complicated simply connected spaces do not satisfy
condition \ref{isH}.  The main result of \cite{CKMVW} shows that the extension
problem is undecidable for even-dimensional spheres, which are the simplest
example.  Other examples include complex projective spaces and most Grassmannians
and Stiefel manifolds.

\subsection{Proof ideas}
Suppose that $Y$ has the rational homotopy type of an H-space through dimension
$d$, but not through dimension $d+1$.  To prove the main theorem, we must provide
an algorithm which computes $[X,Y]^f$ if $\cd(X,A) \leq d$ and decides whether
$[X,Y]^f$ is nonempty if $\cd(X,A)=d+1$.  This builds on work of \v Cadek,
Kr\v c\'al, Matou\v sek, Vok\v r\'inek, and Wagner \cite{CKMVW2}.

To provide an algorithm, we use the rational H-space structure of the $d$th
Postnikov stage $Y_d$ of $Y$.  In this case, we can build an H-space $H$ of
finite type together with rational equivalences
\[H \to Y_d \to H\]
as well as an ``H-space action'' of $H$ on $Y_d$, that is, a map
$\act:H \times Y_d \to Y_d$ which satisfies various compatibility properties.
These ensure that the set $[X/A,H]$ (where $A$ is mapped to the basepoint) acts
via composition with $\act$ on $[X,Y_d]^f$.  In turn, $[X/A,H]$ is a product of
cohomology groups and therefore easily computable, and this allows us to also
compute $[X,Y_d]^f$.  When $\cd(X,A) \leq d$, the obvious map
$[X,Y]^f \to [X,Y_d]^f$ is a bijection; when $\cd(X,A)=d+1$, this map is a
surjection.  This gives the result.

In the last part of the paper, we study the extension problem in the case that
$Y$ is not a rational H-space through dimension $d$ and connect it to Hilbert's
tenth problem.  Recall that Hilbert asked for an algorithm to determine whether a
system of Diophantine equations has a solution.  Work of Davis, Putnam, Robinson,
and Matiyasevich showed that no such algorithm exists.  It turns out that the
problem is still undecidable for very restricted classes of systems of quadratic
equations; this was used in \cite{CKMVW} to show that the extension problem for
maps to $S^{2n}$ is undecidable.  We generalize their work: extension problems
for maps to a given $Y$ are shown to encode systems of Diophantine equations in
which terms are values on vectors of variables of a fixed bilinear map which
depends on $Y$.  We conjecture that Hilbert's tenth problem restricted to any
such subtype is undecidable and prove this in certain special cases.  However,
the general case seems quite difficult; in particular, it would imply a
long-standing conjecture on the undecidability of Hilbert's tenth problem over
number rings.

\subsection{Computational complexity}
Unlike \v Cadek et al.~\cite{CKMVW2,CKV}, whose algorithms are polynomial for
fixed $d$, and like Vok\v r\'inek \cite{Vok}, we do not give any kind of
complexity bound on the run time of the algorithm which computes $[X,Y]^f$.  In
fact, there are several steps in which the procedure is to iterate until we find
a number that works, with no a priori bound on the size of the number, although
it is likely possible to bound it in terms of dimension and other parameters such
as the cardinality of the torsion subgroups in the homology of $Y$.  There is
much space to both optimize the algorithm and discover bounds on the run time.

\subsection{The fiberwise case}
In a paper of \v Cadek, Kr\v c\'al, and Vok\v r\'inek \cite{CKV}, the results of
\cite{CKMVW2} are extended to the \emph{fiberwise} case, that is, to computing
the set of homotopy classes of lifting-extensions completing the diagram
\begin{equation} \label{axyb}
  \begin{gathered}
    \xymatrix{
      A \ar[r]^f \ar@{^(->}[d]_\incl & Y \ar@{->>}[d]^p \\
      X \ar@{-->}[ru] \ar[r]^g & B,
    }
  \end{gathered}
\end{equation}
where $X$ is $2d$-dimensional and the fiber of $Y \xrightarrow{p} B$ is
$d$-connected.  Vok\v r\'inek \cite{Vok} also remarks that his results for
odd-dimensional spheres extend to the fiberwise case.  Is there a corresponding
fiberwise generalization for the results of this paper?  The na\"\i ve hypothesis
would be that $[X,Y]^f_p$ is computable whenever the fiber of
$Y \xrightarrow{p} B$ is a rational H-space through dimension $n$.  This is
false; as demonstrated by the example below, rational homotopy obstructions may
still crop up in the interaction between base and fiber.

The correct fiberwise statement should relate to rational fiberwise H-spaces, as
discussed for example in \cite{LupS}.  However, such a result presents technical
difficulties which will require significant new ideas to overcome.
\begin{ex}
  Let $B=S^6 \times S^2$ and $Y$ be the total space of the fibration
  \[S^7 \to Y \xrightarrow{p_0} B \times (S^3)^2\]
  whose Euler class (a.k.a.\ the $k$-invariant of the corresponding
  $K(\mathbb{Z},7)$-bundle) is
  \[[S^6 \times S^2]+[(S^3)^2 \times S^2] \in H^8(B \times(S^3)^2).\]
  Then the fiber of $p=\pi_1 \circ p_0:Y \to B$ is the H-space
  $(S^3)^2 \times S^7$, but the intermediate $k$-invariant given above has a term
  which is quadratic in the previous part of the fiber.

  Given a system of $s$ polynomial equations each of the form
  \[\sum_{1 \leq i<j \leq r} a_{ij}^{(k)}(x_iy_j-x_jy_i)=b_k,\]
  with variables $x_1,\ldots,x_r,y_1,\ldots,y_r$ and coefficients $b_k$ and
  $a_{ij}^{(k)}$, we form a space $X'$ by taking $\bigvee_r S^3$ and attaching $s$
  $6$-cells, the $k$th one via an attaching map whose homotopy class is
  \[\sum_{1 \leq i<j \leq r} a_{ij}^{(k)}[\id_i,\id_j],\]
  where $\id_i$ is the inclusion map of the $i$th $3$-sphere.  We fix a map
  $f':X' \to S^6$ which collapses the $3$-cells and restricts to a map of degree
  $-b_k$ on the $k$th $6$-cell.  This induces a map $f=f' \times \id$ from
  $X=X' \times S^2$ to $B$.

  A lift of $f$ to $B \times (S^3)^2$ corresponds to an assignment of
  the variables $x_i$ and $y_i$.  The existence of a further lift to $Y$ is then
  equivalent to whether this assignment is a solution to the system of equations
  above.  Since the existence of such a solution is in general undecidable by
  \cite[Lemma 2.1]{CKMVW}, so is the existence of a lift of $f$ through $p$.
\end{ex}
\begin{rmk}
  The role of $S^2$ in this example is to make the fiber into a rational H-space.
  If we let $B=S^6$ and $Y$ be the total space of the fibration
  \[S^5 \to Y \to B \times (S^3)^2\]
  with Euler class $[S^6]+[(S^3)^2]$, then the fiber of $Y \to B$ is no longer a
  product $S^5 \times (S^3)^2$, even rationally, but rather has a nontrivial
  rational $k$-invariant in its Postnikov tower.
\end{rmk}

\subsection{Structure of the paper}
I have tried to make this paper readable to any topologist as well as anyone who
is familiar with the work of \v Cadek et al.  Thus \S\ref{S:RHT} and \ref{S:H}
attempt to introduce all the necessary algebraic topology background which is not
used in \v Cadek et al.'s papers: a bit of rational homotopy theory and some
results about H-spaces.  For the benefit of topologists, I have tried to separate
the ideas that go into constructing a structure on mapping class sets from those
required to compute this structure.  The construction of the group and action in
Theorem \ref{main} is discussed in \S\ref{S:XYA}.  In \S\ref{S:comp}, we
introduce previous results in computational homotopy theory from
\cite{CKMVW2,CKV,FiVo}, and in \S\ref{S:XYAcomp} we use them to compute the
structure we built earlier.  Finally, in \S\ref{S:H10} and \ref{S:undec}, we
discuss Hilbert's tenth problem and its relation to undecidability of the
extension problem.

\subsection*{Acknowledgements}
I would like to thank Shmuel Weinberger for explaining some facts about H-spaces,
and Marek Filakovsk\'y, Luk\'a\v s Vok\v r\'inek, and Uli Wagner for other useful
conversations and encouragement.  I would also like to thank the two referees for
their careful reading and their many corrections and suggestions which have
greatly improved the paper.  The second referee, in particular, caught a major
error which was present in previous versions.  I was partially supported by NSF
grant DMS-2001042.

\section{Rational homotopy theory} \label{S:RHT}

Rational homotopy theory is a powerful algebraicization of the topology of simply
connected topological spaces first introduced by Quillen \cite{Qui} and Sullivan
\cite{SulLong}.  The subject is well-developed, and the texts \cite{GrMo} and
\cite{FHT} are recommended as a comprehensive reference.  This paper requires
only a very small portion of the considerable machinery that has been developed,
and this short introduction should suffice for the reader who is assumed to be
familiar with Postnikov systems and other constructs of basic algebraic topology.

The key topological idea is the construction of rationalized spaces: to any
simply connected CW complex $X$ one can functorially (at least up to homotopy)
associate a space $X_{(0)}$ whose homology (equivalently, homotopy) groups are
$\mathbb{Q}$-vector spaces.\footnote{It's worth pointing out that this fits into
  a larger family of \emph{localizations} of spaces, another of which is used in
  the proof of Lemma \ref{lem:torsion}.}  There are several ways of constructing
such a rationalization, but the most relevant to us is by induction up the
Postnikov tower: the rationalization of a point is a point, and then given a
Postnikov stage
\[\xymatrix{
  K(\pi_n(X),n) \ar[r] & X_n \ar[r] \ar@{->>}[d] & E(\pi_n(X),n+1) \ar@{->>}[d]\\
  & X_{n-1} \ar[r]^-{k_n} & K(\pi_n(X),n+1),
}\]
one replaces it with
\[\xymatrix{
  K(\pi_n(X) \otimes \mathbb{Q},n) \ar[r] & X_{n(0)} \ar[r] \ar@{->>}[d]
  & E(\pi_n(X) \otimes \mathbb{Q},n+1) \ar@{->>}[d] \\
  & X_{n-1(0)} \ar[r]^-{k_n \otimes \mathbb{Q}}
  & K(\pi_n(X) \otimes \mathbb{Q},n+1).
}\]
This builds $X_{n(0)}$ given $X_{n-1(0)}$, and then $X_{(0)}$ is the homotopy type
of the limit of this construction.  We say two spaces are
\emph{rationally equivalent} if their rationalizations are homotopy equivalent.

The second key fact is that the homotopy category of rationalized spaces of
\emph{finite type} (that is, for which all homology groups, or equivalently all
homotopy groups, are finite-dimensional vector spaces) is equivalent to several
purely algebraic categories.  The one most relevant for our purpose is the
Sullivan DGA model.

A \emph{differential graded algebra} (DGA) over $\mathbb{Q}$ is a cochain complex
of $\mathbb{Q}$-vector spaces equipped with a graded commutative multiplication
which satisfies the (graded) Leibniz rule.  A familiar example is the algebra of
differential forms on a manifold.  A key insight of Sullivan was to associate to
every space $X$ of finite type a \emph{minimal} DGA $\mathcal{M}_X$ constructed
by induction on degree as follows:
\begin{itemize}
\item $\mathcal{M}_X(1)=\mathbb{Q}$ with zero differential.
\item For $n \geq 2$, the algebra structure is given by
  \[\mathcal{M}_X(n)=
  \mathcal{M}_X(n+1) \otimes \Lambda\Hom(\pi_n(X);\mathbb{Q}),\]
  where $\Lambda V$ denotes the free graded commutative algebra generated by $V$.
\item The differential is given on the elements of $\Hom(\pi_n(X);\mathbb{Q})$
  (\emph{indecomposables}) by the dual of the $n$th $k$-invariant of $X$,
  \[\Hom(\pi_n(X);\mathbb{Q}) \xrightarrow{k_n^*} H^{n+1}(X;\mathbb{Q}),\]
  and extends to the rest of the algebra by the Leibniz rule.  Although it is
  only well-defined up to a coboundary, this definition makes sense because one
  can show by induction that $H^k(\mathcal{M}_X(n-1))$ is naturally isomorphic to
  $H^k(X_{n-1};\mathbb{Q})$, independent of the choices made in defining the
  differential at previous steps.

  Note that from this definition, it follows that for an indecomposable $y$ of
  degree $n$, $dy$ is an element of degree $n+1$ which can be written as a
  polynomial in the indecomposables of degree $<n$.  In particular, it has no
  linear terms.
\end{itemize}
The DGA $\mathcal{M}_X$ is the functorial image of $X_{(0)}$ under an equivalence
of homotopy categories.

Many topological constructions can thus be translated into algebraic ones.  This
paper will use the following:
\begin{itemize}
\item The Eilenberg--MacLane space $K(\pi,n)$ corresponds to the DGA
  $\Lambda \Hom(\pi,\mathbb{Q})$ with generators concentrated in dimension $n$
  and zero differential.
\item Product of spaces corresponds to tensor product of DGAs.  In particular:
  \begin{prop} \label{prop:Q}
    The following are equivalent for a space $X$:
    \begin{enumerate}[(a)]
    \item $X$ is rationally equivalent to a product of Eilenberg--MacLane spaces.
    \item The minimal model of $X$ has zero differential.
    \item The rational Hurewicz map
      $\pi_*(X) \otimes \mathbb{Q} \to H_*(X;\mathbb{Q})$ is injective.
    \end{enumerate}
  \end{prop}
\end{itemize}
Finally, we note the following theorem of Sullivan:
\begin{thm}[Sullivan's finiteness theorem {\cite[Theorem 10.2(i)]{SulLong}}]
  Let $X$ be a finite complex and $Y$ a simply connected finite complex.  Then
  the map $[X,Y] \to [X,Y_{(0)}]$ induced by the rationalization functor is
  finite-to-one.
\end{thm}
Note that this implies that if the map $Y \to Z$ between finite complexes induces
a rational equivalence, then the induced map $[X,Y] \to [X,Z]$ is also
finite-to-one.

\section{H-spaces} \label{S:H}
A pointed space $(H,o)$ is an H-space if it is equipped with a binary operation
$\add:H \times H \to H$ satisfying $\add(x,o)=\add(o,x)=x$ (the basepoint acts as
an identity).  In addition, an H-space is \emph{homotopy associative} if
\[\add \circ (\add,\id) \simeq \add \circ (\id,\add)\]
and \emph{homotopy commutative} if $\add \simeq \add \circ \tau$, where $\tau$ is
the ``twist'' map sending $(x,y) \mapsto (y,x)$.  We will interchangeably denote
our H-space operations (most of which will be homotopy associative and
commutative) by the usual binary operator $+$, as in $x+y=\add(x,y)$.

A classic result of Sugawara, see \cite[Theorem 3.4]{StaBook}, is that a homotopy
associative H-space which is a connected CW complex automatically admits a
\emph{homotopy inverse} $x \mapsto -x$ with the expected property
$\add(-x,x)=o=\add(x,-x)$.

Examples of H-spaces include topological groups and Eilenberg--MacLane spaces.
If $H$ is simply connected, then it is well-known that it has the rational
homotopy type of a product of Eilenberg--MacLane spaces.  Equivalently, from the
Sullivan point of view, $H$ has a minimal model $\mathcal{M}_H$ with zero
differential; see~\cite[\S12(a) Example 3]{FHT} for a proof.  On the other hand,
a product of H-spaces is clearly an H-space.  Therefore we can add ``$X$ is
rationally equivalent to an H-space'' to the list of equivalent conditions in
Prop.~\ref{prop:Q}.  We will generally use the sloppy phrase ``$X$ is a rational
H-space'' to mean the same thing.

It is easy to see that an H-space operation plays nice with the addition on
higher homotopy groups.  That is:
\begin{prop} \label{additive}
  Let $(H,o,\add)$ be an H-space.  Given $f,g:(S^n,*) \to (H,o)$,
  \[[f]+[g]=[\add \circ (f,g)] \in \pi_n(H,o).\]
\end{prop}
Another important and easily verified fact is the following:
\begin{prop}
  If $(H,o,\add)$ is a homotopy associative H-space, then for any pointed space
  $(X,*)$, the set $[X,H]$ forms a group, with the operation given by
  $[\ph]\cdot[\psi]=[\add \circ (\ph,\psi)]$.  If $H$ is homotopy commutative,
  then this group is likewise commutative.

  Moreover, suppose that $H$ is homotopy commutative, and let $A \to X$ be a
  cofibration (such as the inclusion of a CW subcomplex), and $f:A \to H$ a map
  with an extension $\tilde f:X \to H$.  Then the set $[X,H]^f$ of extensions of
  $f$ forms an abelian group with operation given by
  \[[\ph]+[\psi]=[\ph+\psi-\tilde f].\]
\end{prop}

Throughout the paper, we denote the ``multiplication by $r$'' map
\[\underbrace{\id + \cdots + \id}_{r\text{ times}}:H \to H\]
by $\chi_r$.  The significance of this map is in the following lemmas, which we
will repeatedly apply to various obstruction classes:
\begin{lem} \label{lem:torsion}
  Let $H$ be an H-space of finite type, $A$ be a finitely generated coefficient
  group, and let $\alpha \in H^n(H;A)$ be a cohomology class of finite order.
  Then there is an $r>0$ such that $\chi_r^*\alpha=0$.
\end{lem}
In other words, faced with a finite-order obstruction, we can always get rid of
it by precomposing with a multiplication map.  Before giving the proof, we
develop a bit more of the theory:
\begin{lem} \label{lem:non-torsion}
  Let $H$ be a simply connected H-space of finite type.  Then for every $r>0$,
  \[\chi_r^*(H^*(H;\mathbb{Z})) \subseteq rH^*(H;\mathbb{Z})+\text{torsion}.\]
\end{lem}
\begin{proof}
  By Prop.~\ref{additive}, $\chi_r$ induces multiplication by $r$ on $\pi_n(H)$.
  Therefore by Prop.~\ref{prop:Q}(c), it induces multiplication by $r$ on the
  indecomposables of the minimal model $\mathcal{M}_H$.  Therefore it induces
  multiplication by some $r^k$ on every class in $H^n(H;\mathbb{Q})$.
\end{proof}
Combining the two lemmas gives us a third:
\begin{lem} \label{lem:both}
  Let $H$ be a simply connected H-space of finite type and $A$ a finitely
  generated coefficient group.  Then for any $r>0$ and any $n>0$, there is an
  $s>0$ such that
  \[\chi_s^*(H^n(H);A) \subseteq rH^n(H;A).\]
\end{lem}
\begin{proof}[Proof of Lemma \ref{lem:torsion}.]
  I would like to thank Shmuel Weinberger for suggesting this proof.

  Let $q$ be the order of $\alpha$.  By Prop.~\ref{additive}, for $f:S^k \to H$,
  $(\chi_q)_*[f]=q[f]$.

  Let $H[1/q]$ be the universal cover of the mapping torus of $\chi_q$; this
  should be thought of as an infinite mapping telescope.  By the above, the
  homotopy groups of $H[1/q]$ are $\mathbb{Z}[1/q]$-modules (the telescope
  localizes them away from $q$).  This implies, by \cite[Thm.~2.1]{SulLoc}, that
  the reduced homology groups are also $\mathbb{Z}[1/q]$-modules.  To understand
  the cohomology groups, we use the exact sequence
  \[0 \to \Ext(H_{n-1}(H[1/q]),A) \to H^n(H[1/q];A) \to \Hom(H_{n-1}(H[1/q]),A) \to
  0\]
  coming from the universal coefficient theorem.  If $M$ is a
  $\mathbb{Z}[1/q]$-module, then so is $\Hom(M,G)$ for any abelian group $G$: for
  any homomorphism $h$, we take $[h/q](m)=h(m/q)$.  Since $\Ext(M,G)$ is the
  first homology of the chain complex
  \[0 \to \Hom(M,I^0) \to \Hom(M,I^1) \to \cdots,\]
  where $I^*$ is an injective resolution of $G$, it is also a
  $\mathbb{Z}[1/q]$-module.  It follows that $H^n(H[1/q];A)$ is a
  $\mathbb{Z}[1/q]$-module.  Now, by the Milnor exact sequence \cite{MilAHT}, the
  map
  \[H^n(H[1/q];A) \to \varprojlim\bigl(\cdots \xrightarrow{\chi_q^*} H^n(H;A)
  \xrightarrow{\chi_q^*} H^n(H;A)\bigr)\]
  is surjective, and hence this inverse limit is also a $\mathbb{Z}[1/q]$-module.

  Now we would like to show that for some $t$, $(\chi_q^*)^t\alpha=0$, so that we
  can take $r=q^t$.  Suppose not, so that $(\chi_q^*)^t\alpha$ is nonzero for
  every $t$.  Clearly every element in the sequence
  \[\alpha,\chi_q^*\alpha,(\chi_q^*)^2\alpha,\ldots\]
  has order which divides $q$; moreover, since there are finitely many such
  elements, the sequence eventually cycles.  Extrapolating this cycle backward
  gives us a nonzero element of the inverse limit above, which likewise has order
  dividing $q$.  This contradicts the fact that this inverse limit is a
  $\mathbb{Z}[1/q]$-module.
\end{proof}
Note that this proof does not produce an effective bound on $t$.  This prevents
our algorithmic approach from yielding results that are as effective as those of
Vok\v{r}\'inek in \cite{Vok}.

We will also require the similar but more involved fact.
\begin{lem} \label{lem:product}
  Let $(H,o,\add)$ be a simply connected H-space of finite type, $U$ another
  space of finite type, $A$ a finitely generated coefficient group, and $n>0$.
  \begin{enumerate}[(i)]
  \item Suppose that $\alpha \in H^n(H \times U,o \times U;A)$ is torsion.  Then
    there is an $r>0$ such that $(\chi_r,\id)^*\alpha=0$.
  \item Let $\alpha \in H^n(H \times U,o \times U;\mathbb Z)$.  Then for every
    $r>0$,
    \[(\chi_r,\id)^*\alpha \in rH^n(H \times U,o \times U;\mathbb Z)+\text{torsion}.\]
  \item For every $r>0$ there is an $s>0$ such that
    \[(\chi_s,\id)^*H^n(H \times U,o \times U;A) \subseteq rH^n(H \times U,o \times U;A).\]
  \end{enumerate}
\end{lem}
\begin{proof}
  Let where $i_2:U \to H \times U$ is the inclusion $u \mapsto (o,u)$.  We first
  note that since the map $i_2^*$ on cohomology is a surjection in every degree,
  $H^n(H \times U,o \times U;A)=\ker i_2^*$.  Thus we can equivalently prove
  parts (i) and (ii) for an $\alpha \in H^n(H \times U;A)$ such that
  $i_2^*\alpha=0$.  We use several not-quite-standard algebraic topology facts
  which can be found in \cite[\S5.5]{Spa}.

  We first consider $A=\mathbb Z$.  For this we use the K\"unneth formula for
  cohomology, which gives a natural short exact sequence
  \begin{equation} \label{Kunneth}
    0 \to \bigoplus_{k+\ell=n} H^k(H) \otimes H^\ell(U) \to H^n(H \times U) \to
    \bigoplus_{k+\ell=n+1} \Tor(H^k(H),H^\ell(U)) \to 0.
  \end{equation}
  To demonstrate (i), we will first show that we can choose an $r_0$ such that
  $(\chi_{r_0},\id)^*\alpha$ is in the image of
  $\bigoplus_{k+\ell=n} H^k(H) \otimes H^\ell(U)$; in other words, such that the
  projection of $(\chi_{r_0},\id)^*\alpha$ to
  $\bigoplus_{k+\ell=n+1} \Tor(H^k(H),H^\ell(U))$ is zero.  To see this, recall that
  for cyclic groups $A$ and $B$, $\Tor(A,B)$ is trivial unless both $A$ and $B$
  are finite, and that the $\Tor$ functor distributes over direct sum.  Therefore
  $\Tor(H^k(H),H^\ell(U))$ is generated by elementary tensors $\eta \otimes \nu$
  where $\eta \in H^k(H)$ and $\nu \in H^\ell(U)$ are torsion elements.  By Lemma
  \ref{lem:torsion}, for each such elementary tensor, we can pick $r(\eta)$ such
  that $\chi_{r(\eta)}^*\eta=0$ and therefore
  \[(\chi_{r(\eta)},\id)^*(\eta \otimes \nu)=0 \in \Tor(H^k(H),H^\ell(U)).\]
  We then choose $r_0$ to be the least common multiple of all the $r(\eta)$'s.

  Now fix a decomposition of each $H^k(H)$ and $H^\ell(U)$ into cyclic factors to
  write $(\chi_{r_0},\id)^*\alpha$ as a sum of elementary tensors.  Since
  $i_2^*\alpha=0$, $(\chi_{r_0},\id)^*\alpha$ has no summands of the form
  $1 \otimes u$; moreover, each summand is itself torsion.  For every other
  elementary tensor $h \otimes u$, we can use Lemma \ref{lem:torsion} (if $h$ is
  torsion) or Lemma \ref{lem:both} (otherwise, since then $u$ is torsion) to find
  an $s(h,u)$ such that $\chi_{s(h,u)}^*h \otimes u=0$.

  Finally, we can take $r$ to be the product of $r_0$ with the least common
  multiple of the $s(h,u)$'s.  This completes the proof of (i) for $A=\mathbb Z$.

  To see (ii), we use the fact that the K\"unneth sequence \eqref{Kunneth}
  splits, albeit non-naturally.  Therefore, since we
  are ignoring torsion, we may assume
  $\alpha \in \bigoplus_{k+\ell=n} H^k(H) \otimes H^\ell(U)$.   Applying Lemma
  \ref{lem:non-torsion} to $H^k(H)$ for all $0<k<n$, we get the result.

  Finally, (iii) with integer coefficients follows from (i) and (ii).

  Now we need to handle other coefficient groups.  We can assume $A$ is a finite
  cyclic group, since everything we need commutes with direct sums.  For this
  case we use a version of the universal coefficient theorem which states that
  \[0 \to H^n(H \times U) \otimes A \to H^n(H \times U;A) \to
  \Tor(H^{n+1}(H \times U),A) \to 0\]
  is an exact sequence.  Let $\alpha \in H^n(H \times U,o \times U;A)$ be
  torsion.  We use the same outline as before to show that (i) holds.  First we
  see that there is an $r_0$ such that $(\chi_{r_0},\id)^*\alpha$ is in the kernel
  of the map to $\Tor(H^{n+1}(H \times U),A)$; this follows from the integral
  case of (i) and the fact that $\Tor(G,A)$ contains only the $A$-torsion
  elements of $G$.  Next we see that the preimage of $(\chi_{r_0},\id)^*\alpha$ in
  $H^n(H \times U) \otimes A$ is also annihilated by some $(\chi_{r_1},\id)^*$;
  this follows from the integral case of (iii).  Then
  $(\chi_{r_0r_1},\id)^*\alpha=0$.

  The general case of (iii) again follows from (i) and (ii).
\end{proof}

\section{The algebraic structure of $[X,Y]^f$} \label{S:XYA}

We start by constructing the desired structure on $[X,Y]^f$ when $Y$ is a
rational H-space.  From the previous section, such a $Y$ is rationally equivalent
to a product of Eilenberg--MacLane spaces.  In particular, it is rationally
equivalent to $H=\prod_{n=2}^\infty K(\pi_n(Y),n)$, which we give the product
H-space structure.  We will harness this to prove the following result.
\begin{thm} \label{thm:XYA}
  Suppose that $Y$ is a rational H-space through dimension $d$, denote by $Y_d$
  the $d$th Postnikov stage of $Y$, and let $H_d=\prod_{n=2}^\infty K(\pi_n(Y),n)$.
  Suppose $(X,A)$ is a finite simplicial pair and $f:A \to Y$ a map.  Then
  $[X,Y_d]^f$ admits a virtually free and faithful action by $[X,H_d]^f$ induced
  by a map $H_d \to Y_d$.
\end{thm}
The proof of this theorem occupies the rest of the section.  Later, in
\S\ref{S:XYAcomp}, we give an algorithm for computing this action which closely
mirrors this proof.  Before beginning the proof of Theorem \ref{thm:XYA}, we see
how such an algorithm would also provide the algorithms whose existence is
asserted in Theorem \ref{main}.

If $(X,A)$ has cohomological dimension $d+1$, then there is no obstruction to
lifting an extension $X \to Y_d$ of $f$ to $Y$, as the first obstruction lies in
$H^{d+2}(X,A;\pi_{d+1}(Y)) \cong 0$.  Therefore $[X,Y]^f$ is nonempty if and only
if $[X,Y_d]^f$ is nonempty.

If $(X,A)$ has cohomological dimension $d$, then in addition every such lift is
unique: the first obstruction to homotoping two lifts lies in
$H^{d+1}(X,A;\pi_{d+1}(Y)) \cong 0$.  Therefore $[X,Y]^f \cong [X,Y_d]^f$.

\subsection{An H-space action on $Y_n$}
Denote the $n$th Postnikov stages of $Y$ and $H$ by $Y_n$ and $H_n$,
respectively, and the H-space zero and multiplication on $H_n$ by $o_n$ and by
$+$ or $\add_n:H_n \times H_n \to H_n$.  We will inductively construct the
following additional data:
\begin{enumerate}[(i)]
\item Maps $H_n \xrightarrow{\HtoY_n} Y_n \xrightarrow{\YtoH_n} H_n$ inducing
  rational equivalences such that $\YtoH_n\HtoY_n$ is homotopic to the
  multiplication map $\chi_{r_n}$ for some integer $r_n$.
\item A map $\act_n:H_n \times Y_n \to Y_n$ defining an \emph{H-space action},
  (that is such that $\act_n(\zerosec_n,x)=x$ and the diagram
  \begin{equation} \label{H:action}
    \begin{gathered}
    \xymatrixcolsep{3pc}
    \xymatrix{
      H_n \times H_n \times Y_n \ar[r]^-{(\add_n,\id)} \ar[d]^{(\id,\act_n)} &
      H_n \times Y_n \ar[d]^{\act_n} \\
      H_n \times Y_n \ar[r]^-{\act_n} & Y_n
    }
    \end{gathered}
  \end{equation}
  commutes up to homotopy) which is ``induced by $\HtoY_n$'' in the sense of the
  homotopy commutativity of
  \begin{equation} \label{H:xr}
  \begin{gathered}
    \xymatrix{
      H_n \times H_n \ar[r]^{(\id,\HtoY_n)} \ar[d]^{\add_n} &
      H_n \times Y_n \ar[r]^{(\chi_{r_n},\YtoH_n)} \ar[d]^{\act_n} &
      H_n \times H_n \ar[d]^{\add_n} \\
      H_n \ar[r]^{\HtoY_n} & Y_n \ar[r]^{\YtoH_n} & H_n.
    }
  \end{gathered}
  \end{equation}
\end{enumerate}
Note that when we pass to rationalizations, the existence of such a structure is
obvious: one takes $\HtoY_{n(0)}$ to be the identity, $\act_{n(0)}=\add_{n(0)}$, and
$\YtoH_{n(0)}$ to be multiplication by $r_n$.

\subsection{The action of $[X/A,H_d]$ on $[X,Y_d]^f$} \label{S:VFF}
Now suppose that we have constructed the above structure.  Then $\add_d$ induces
the structure of a finitely generated abelian group on the set $[X/A,H_d]$, which
we identify with the set of homotopy classes of maps $X \to H_d$ sending $A$ to
$\zerosec_d \in H_d$.  Moreover, this group acts on $[X,Y_d]^f$ via the action
$[\ph]\cdot[\psi]=[\act_d\circ(\ph,\psi)]$.

It remains to show that this action is virtually free and faithful.  Indeed,
notice that pushing this action forward along $\YtoH_d$ gives the action of
of $[X/A,H_d]$ on $[X,H_d]^{\YtoH_df}$ via $[\ph] \cdot [\psi]=r_d[\ph]+[\psi]$,
which is clearly virtually free and faithful.  This implies that the action on
$[X,H_d]^{v_df}$ is virtually free.  Moreover, the map
$\YtoH_d \circ{}:[X,Y_d]^f \to [X,H_d]^{\YtoH_df}$ is finite-to-one by Sullivan's
finiteness theorem.  Thus the action on $[X,Y_d]^f$ is also virtually faithful.

\subsection{The Postnikov induction} \label{S:2.2}
Now we construct the H-space action.  For $n=1$ all the spaces are points and all
the maps are trivial.  So suppose we have constructed the maps $\HtoY_{n-1}$,
$\YtoH_{n-1}$, and $\act_{n-1}$, and let $k_n:Y_{n-1} \to K(\pi_n(Y),n+1)$ be the
$n$th $k$-invariant of $Y$.  For the inductive step, it suffices to prove the
following lemma:
\begin{lem}
  There is an integer $q>0$ such that we can define $\HtoY_n$ to be a lift of
  $\HtoY_{n-1}\chi_q$, and construct $\YtoH_n$ and a solution
  $\act_n:H_n \times Y_n \to Y_n$ to the homotopy lifting-extension problem
  \begin{equation} \label{XYA:a}
    \begin{gathered}
      \xymatrix{
        H_n \times H_n \ar[d]_{(\id, \HtoY_n)} \ar[rr]^-{\add_n}
        && H_n \ar[r]^{\HtoY_n} & Y_n \ar@{->>}[d] \\
        H_n \times Y_n \ar@{-->}[rrru]|-{\act_n} \ar[r]_{(\chi_q,\id)}
        & H_n \times Y_n \ar@{->>}[r]
        & H_{n-1} \times Y_{n-1} \ar[r]^-{\act_{n-1}} & Y_{n-1}
      }
    \end{gathered}
  \end{equation}
  so that the desired conditions are satisfied.
\end{lem}
\begin{proof}
  First, since $Y$ is rationally a product, $k_n$ is of finite order, so by Lemma
  \ref{lem:torsion} there is some $q_0$ such that $k_n\HtoY_{n-1}\chi_{q_0}=0$, and
  therefore
  \[\xymatrixcolsep{3.5pc}
  \xymatrix{
    H_n \ar@{->>}[d] \ar[r]^{\hatHtoY} & Y_n \ar@{->>}[d] \\
    H_{n-1} \ar[r]^{\HtoY_{n-1}\chi_{q_0}} & Y_{n-1};
  }\]
  is a pullback square.  We will define $\HtoY_n=\hatHtoY\chi_{q_2q_1}$, with $q_1$
  and $q_2$ to be determined and $q=q_2q_1q_0$.

  Now we construct $\act_n$.  Given a map $f$, we write $M(f)$ to mean its
  mapping cylinder, and let
  \[\Hact_{n-1}:H_{n-1} \times M(\HtoY_{n-1}) \to Y_{n-1}\]
  be a map which restricts to $\act_{n-1}$ on $H_{n-1} \times Y_{n-1}$ and
  $\add_{n-1}$ on $H_{n-1} \times H_{n-1}$.  Such a map exists because \eqref{H:xr}
  holds in degree $n-1$.  We will construct a lifting-extension
  \[\xymatrix{
    (H_n \times H_n) \cup (o_n \times M(\hatHtoY))
    \ar@{>->}[d] \ar[rr]^-{[\add_n \circ (\chi_{q_1},\id)] \cup \id}
    && M(\hatHtoY) \ar[r]^{\text{project}} & Y_n \ar@{->>}[d] \\
    H_n \times M(\hatHtoY)
    \ar@{-->}[rrru]|-{\widehat\Hact} \ar[r]_{(\chi_{q_1q_0},\id)}
    & H_n \times M(\hatHtoY) \ar@{->>}[r]
    & H_{n-1} \times M(\HtoY_{n-1}) \ar[r]^-{\Hact_{n-1}} & Y_{n-1}
  }\]
  It is easy to see that then for any $q_2>0$,
  \[\act_n=(\widehat\Hact|_{H_n \times Y_n}) \circ (\chi_{q_2},\id)\]
  satisfies \eqrefb{XYA:a}.  Moreover, then the desired identity
  $\act_n(o_n,x)=x$ is automatically satisfied.

  Note that the outer rectangle commutes since we know \eqref{H:xr} holds in
  degree $n-1$.  Now, write
  \begin{align*}
    A &= H_n \times M(\hatHtoY) \\
    B &= (H_n \times H_n) \cup (o_n \times M(\hatHtoY)) \\
    C &= o_n \times M(\hatHtoY).
  \end{align*}
  Since $\hatHtoY$ is a rational equivalence, so are the inclusions of
  $H_n \times H_n$ into $A$ and $B$. Therefore, the obstruction
  $\mathcal{O} \in H^{n+1}(A,B;\pi_n(Y))$ to finding the lifting-extension is of
  finite order.  We will show that when $q_1$ is large enough, this obstruction
  is zero.

  The obstruction group fits into the exact sequence of the triple $(A,B,C)$:
  \[\cdots \to H^n(B,C;\pi_n(Y)) \xrightarrow{\delta}
  H^{n+1}(A,B;\pi_n(Y)) \xrightarrow{\rel^*}
  H^{n+1}(A,C;\pi_n(Y)) \to \cdots,\]
  and so the image $\rel^*\mathcal{O}$ in $H^{n+1}(A,C;\pi_n(Y))$ is torsion.  By
  Lemma \ref{lem:product}(i), that means that
  $(\chi_s,\id)^*(\rel^*\mathcal{O})=0$ for some $s>0$.

  Now we look at a preimage under $\delta$ of $(\chi_s,\id)^*\mathcal{O}$, which
  we call $\alpha \in H^n(B,C;\pi_n(Y))$.  By excision,
  \[H^n(B,C;\pi_n(Y)) \cong H^n(H_n \times H_n,o_n \times H_n;\pi_n(Y)).\]
  Applying Lemma \ref{lem:product}(iii), we can find a $t$ such that
  $\chi_t^*\alpha \in \ker\delta$ and therefore
  \[\delta((\chi_t,\id)^*\alpha)=(\chi_{st},\id)^*\mathcal{O}=0.\]
  Thus for $q_1=st$, we can find a map $\widehat\Hact$ completing the diagram.

  Now we ensure that \eqref{H:action} commutes by picking an appropriate $q_2$.
  Define $\widehat\act=\widehat\Hact|_{H_n \times Y_n}$; then the diagram
  \[\xymatrixcolsep{3pc}
  \xymatrix{
    H_n \times H_n \times Y_n \ar[r]^-{(\add_n,\id)} \ar[d]^{(\id,\widehat\act)} &
    H_n \times Y_n \ar[d]^{\widehat\act} \\
    H_n \times Y_n \ar[r]^-{\widehat\act} & Y_n
  }\]
  commutes after rationalization.  Since \eqref{H:action} commutes in degree
  $n-1$, the sole obstruction to homotopy commutativity is a torsion class in
  $H^n(H_n\times H_n\times Y;\pi_n(Y_n))$.  Therefore we can again apply Lemma
  \ref{lem:product}(i), this time with $H=H_n \times H_n$ and $U=Y_n$, to find a
  $q_2$ which makes the obstruction zero.

  All that remains is to define $\YtoH_n$.  But we know that $\HtoY_n$ is
  rationally invertible, and so we can find some $\YtoH_n$ such that
  $\YtoH_n\HtoY_n$ is multiplication by some $r_n$.  Moreover, for any such
  $\YtoH_n$, the right square of \eqref{H:xr} commutes up to finite order.  Thus
  by increasing $r_n$ (that is, replacing $\YtoH_n$ by $\chi_{\hat r}\YtoH_n$ for
  some $\hat r>0$) we can make it commute up to homotopy.
\end{proof}

\section{Building blocks of homotopy-theoretic computation} \label{S:comp}
We now turn to describing the algorithms for performing the computations outlined
in the previous two sections.  This relies heavily on machinery and results from
\cite{CKMVW2}, \cite{CKV}, and \cite{FiVo} as building blocks, which in turn rely
on building blocks from the work of Rubio, Sergeraert, and others
\cite{Serg,RuSergSurvey,RuSerg}.  This section is dedicated to explaining these
building blocks.

Our spaces are stored as simplicial sets \emph{with effective homology}.  Roughly
speaking this means a computational black box equipped with:
\begin{itemize}
\item A way to refer to individual simplices and compute their face and
  degeneracy operators.  This allows us to, for example, represent a function
  from a finite simplicial complex or simplicial set to a simplicial set with
  effective homology.
\item A \emph{fully effective} chain complex with a chain homotopy equivalence
  to this set.  We do not need to make this completely precise, but for example
  it allows one to compute the homology and cohomology in any degree and with
  respect to any finitely generated coefficient group, and to know both their
  isomorphism type and (co)chains representing individual classes.
\end{itemize}
This is easy to construct for finite simplicial complexes.  But effective
homology is designed to work with simplicial sets that can be described
algorithmically but are not necessarily finite; in our case, these are finite
Postnikov stages of spaces of finite type.  We refer to \cite{RuSergSurvey} for a
more detailed overview.

Now we summarize the operations which are known to be computable from previous
work.
\begin{thm} \label{blocks}
  \begin{enumerate}[(a)]
  \item \label{K(pi,n)} Given a finitely generated abelian group $\pi$ and
    $n \geq 2$, a model of the Eilenberg--MacLane space $K(\pi,n)$ can be
    represented as a simplicial set with effective homology and a computable
    simplicial group operation.  Moreover, there are algorithms implementing a
    chain-level bijection between $n$-cochains in a finite simplicial complex or
    simplicial set $X$ with coefficients in $\pi$ and maps from $X$ to $K(\pi,n)$
    (the observation dates back to at least \cite{Serg}, but see
    \cite[\S3.7]{CKMVW2} or \cite[\S7.5]{RuSerg} for a detailed explanation).
  \item \label{product} Given a finite family of simplicial sets with effective
    homology, their product can be represented as a simplicial set with effective
    homology (see \cite[\S8.2]{RuSerg} or \cite[\S3.1]{CKMVW2}).
  \item Given a simplicial map $f:X \to Y$ between simplicial sets with effective
    homology, there is a way of representing the mapping cylinder $M(f)$ as a
    simplicial set with effective homology.  (In \cite{CKV} this is remarked to
    be ``very similar to but easier than Prop.~5.11''; the related algebraic
    mapping cylinder construction is done explicitly in e.g.~\cite[\S3]{RER}.)
  \item \label{MP} Given a map $p:Y \to B$, we can compute the $n$th stage of the
    Moore--Postnikov tower for $p$, in the form of a sequence of Kan fibrations
    between simplicial sets with effective homology \cite[Theorem 3.3]{CKV}
    (cf.~\cite[Theorem 1.2]{CKMVW2} for the non-relative version).
  \item \label{prop3.7} Given a diagram
    \[\xymatrix{
      A \ar@{>->}[d] \ar[r] & P_n \ar@{->>}[d] \\
      X \ar[r] \ar@{-->}[ru] & P_{n-1}
    }\]
    where $P_n \to P_{n-1}$ is a step in a (Moore--)Postnikov tower as above,
    there is an algorithm to decide whether a diagonal exists and, if it does,
    compute one \cite[Prop.~3.7]{CKV}.
  \item \label{pullback} Given a fibration $p:Y \to B$ of simply connected
    simplicial complexes and a map $f:X \to B$, we can compute any finite
    Moore--Postnikov stage of the pullback of $p$ along $f$
    \cite[Addendum 3.4]{CKV}.
  \item \label{homotopy} Given a diagram
    \[\xymatrix{
      A \ar[r]^f \ar@{>->}[d]_\incl & Y \ar@{->>}[d]^p \\
      X \ar@{-->}[ru] \ar[r]^g & B,
    }\]
    where $A$ is a subcomplex of a finite complex $X$ and $p$ is a fibration of
    simply connected complexes of finite type, we can compute whether two maps
    $u,v:X \to Y$ completing the diagram are homotopic relative to $A$ and over
    $B$ \cite[see ``Equivariant and Fiberwise Setup'']{FiVo}.
  \item \label{rel-finite} Given a diagram
    \[\xymatrix{
      A \ar[r]^f \ar@{>->}[d]_\incl & Y \ar[d]^p \\
      X \ar@{-->}[ru] \ar[r]^g & B
    }\]
    where $A$ is a subcomplex of a finite complex $X$, $Y$ and $B$ are simply
    connected, and $p$ has finite homotopy groups, we can compute the (finite and
    perhaps empty) set $[X,Y]^f_p$ of homotopy classes of maps completing the
    diagram up to homotopy.
  \end{enumerate}
\end{thm}
\begin{proof}
  We prove only the part which is not given a citation in the statement.


  \subsection*{Part \ref*{rel-finite}}
  Let $d=\dim X$.  One starts by computing the $d$th stage of the
  Moore--Postnikov tower of $p:Y \to B$ using \ref{MP}.  From there, we induct on
  dimension.  At the $k$th step, we have computed the (finite) set of lifts to
  the $k$th stage $P_k$ of the Moore--Postnikov tower.  For each such lift, we
  use \ref{prop3.7} to decide whether it lifts to the $(k+1)$st stage, and
  compute a lift $u:X \to P_{k+1}$ if it does.  Then we compute all lifts by
  computing representatives of each element of $H^{k+1}(X,A;\pi_{k+1}(p))$ and
  modifying $u$ by each of them.  Finally, we use \ref{homotopy} to decide
  which of the maps we have obtained are duplicates and choose one representative
  for each homotopy class in $[X,P_{k+1}]^f_p$.  We are done after step $d$
  since $[X,P_d]^f_p \cong [X,Y]^f_p$.
\end{proof}

\section{Computing $[X,Y]^f$} \label{S:XYAcomp}

We now explain how to compute the group and action described in \S\ref{S:XYA}.
We work with a representation of $(X,A)$ as a finite simplicial set and a
Postnikov tower for $Y$, and perform the induction outlined in that section to
compute $[X,Y_d]^f$ for a given dimension $d$.  The algorithm verifies that $Y$
is indeed a rational H-space through dimension $d$; however, it assumes that $Y$
is simply connected and returns nonsense otherwise.

\subsection{Setup}
Let $d$ be such that $Y_d$ is a rational H-space.  Since the homotopy groups of
$Y$ can be computed, we can use Theorem \ref{blocks}\ref*{K(pi,n)} and
\ref*{product} to compute once and for all the space
\[H_d=\prod_{n=2}^d K(\pi_n(Y),n),\]
and the binary operation $\add_d:H_d \times H_d \to H_d$ is given by the product
of the simplicial group operations on the individual $K(\pi_n(Y),n)$'s.  The
group of homotopy classes $[X/A,H_d]$ is naturally isomorphic to
$\prod_{n=2}^d H^n(X,A;\pi_n(Y))$, making this also easy to compute.  Finally,
given an element of this group expressed as a word in the generators, we can
compute a representative map $X \to H_d$, constant on $A$, by generating the
corresponding cochains of each degree on $(X,A)$ and using them to build maps to
$K(\pi_n(Y),n)$.

We then initialize the induction which will compute maps $\HtoY_d$, $\YtoH_d$,
and $\act_d$ and an integer $r_d$ satisfying the conditions of \S\ref{S:XYA}.
Since $H_1=Y_1$ is a point, we can set $r_1=1$ and $\HtoY_1$, $\YtoH_1$, and
$\act_1$ to be the trivial maps.

\subsection{Performing the Postnikov induction} \label{S:induction}
The induction is performed as outlined in \S\ref{S:2.2}, although we have to be
careful to turn homotopy lifting and extension problems into genuine ones.
Suppose that maps $\HtoY_{n-1}$, $\YtoH_{n-1}$, and $\act_{n-1}$ as desired have
been constructed, along with a map
\[\Hact_{n-1}:H_n \times M(\HtoY_{n-1}) \to Y_{n-1}\]
which restricts to $\add_{n-1}$ on $H_{n-1} \times H_{n-1}$ and $\act_{n-1}$ on
$H_{n-1} \times Y_{n-1}$.  There are five steps to constructing the maps in the
$n$th step:
\begin{enumerate}[1.]
\item Find $q_0$ such that $\HtoY_{n-1}\chi_{q_0}$ lifts to a map
  $\hatHtoY:H_n \to Y_n$, and fix such a map.
\item Find $q_1$ such that the diagram
  \[\xymatrix{
    (H_n \times H_n) \cup (o_n \times M(\hatHtoY))
    \ar@{>->}[d] \ar[rr]^-{[\add_n \circ (\chi_{q_1},\id)] \cup \id}
    && M(\hatHtoY) \ar[r]^{\text{project}} & Y_n \ar@{->>}[d] \\
    H_n \times M(\hatHtoY)
    \ar@{-->}[rrru]|-{\widehat\Hact} \ar[r]_{(\chi_{q_1q_0},\id)}
    & H_n \times M(\hatHtoY) \ar@{->>}[r]
    & H_{n-1} \times M(\HtoY_{n-1}) \ar[r]^-{\Hact_{n-1}} & Y_{n-1}
  }\]
  has a lifting-extension $\widehat\Hact$ along the dotted arrow, and fix such a
  map.
\item Find $q_2$ such that $\widehat\Hact|_{H_n \times Y_n} \circ (\chi_{q_2},\id)$
  makes the diagram \eqref{H:action} commute up to homotopy.  Now we can define
  \begin{align*}
    \HtoY_n &: H_n \to Y_n & \text{by} && \HtoY_n &= \hatHtoY\chi_{q_1q_2}; \\
    \Hact_n &: H_n \times M(u_n) \to Y_n & \text{by}&&
    \Hact_n &= \widehat\Hact \circ (\chi_{q_2},\id \cup \chi_{q_1q_2}); \\
    \act_n &: H_n \times Y_n \to Y_n &
    \text{by} && \act_n &= \Hact_n|_{H_n \times Y_n}.
  \end{align*}
\item Find $q_3$ so that the diagram
  \[\xymatrix{H_n \ar@{>->}[r] \ar@/_1pc/[rr]_{\chi_{q_3}}
  & M(\HtoY_n) \ar@{-->}[r] & H_n}\]
  can be completed by some $\hatYtoH:M(u_n) \to H_n$, and fix such a map.
\item Find $q_4$ so that setting
  \[\YtoH_n=\hatYtoH\chi_{q_4} \quad\text{and}\quad
  r_n=r_{n-1}q_0q_1q_2q_3q_4\]
  makes the diagram \eqref{H:xr} commute up to homotopy.
\end{enumerate}
The first step is done by determining the order of the $k$-invariant
$k_n \in H^{n+1}(Y_{n-1};\pi_n(Y))$.  If this order is infinite, then $Y$ is not
rationally a product of Eilenberg--MacLane spaces, and the algorithm returns
failure.  Otherwise $q_0$ is guaranteed to exist, and we can compute it by
iterating over multiples of the order until we find one that works.

The rest of the steps are guaranteed to succeed for some value of $q_i$, and each
of the conditions can be checked using the operations of Theorem \ref{blocks},
so this part can also be completed by iterating over all possible values until we
find one that works.

\subsection{Computing the action} \label{S:XYAcomp:action}
Let $G=[X/A,H_d]$; we now explain how to compute $[X,Y]^f$ as a set with a
virtually free and faithful action by $G$.

First we must decide whether there is a map $X \to H_d$ extending
$\YtoH_df:A \to H_d$.  If the set $[X,Y_d]^f$ has an element $e$, then $\YtoH_df$
has an extension $\YtoH_de$, so if we find that there is no such extension, we
return the empty set.  Otherwise we compute such an extension $\psi_0$.

\begin{lem} \label{lem:extends}
  We can determine whether an extension $\psi_0:X \to H_d$ of $\YtoH_df$ exists,
  and compute one if it does.
\end{lem}
\begin{proof}
  Recall that $H_d=\prod_{n=2}^d K(\pi_n(Y),n)$.  Write $\proj_n$ for the
  projection to the $K(\pi_n(Y),n)$ factor.  Then the extension we desire exists
  if and only if for each $n<d$, the cohomology class in $H^n(A;\pi_n(Y))$
  represented by $\proj_n\YtoH_df$ has a preimage in $H^n(X;\pi_n(Y))$ under the
  map $\incl^*$.

  We look for an explicit cocycle $\sigma_n \in C^n(X;\pi_n(Y))$ whose
  restriction to $A$ is $\proj_n\YtoH_df$.  We can compute cycles which generate
  $H^n(X;\pi_n(Y))$ (because $X$ has effective homology) as well as generators
  for $\delta C^{n-1}(X;\pi_n(Y))$ (the coboundaries of individual
  $(n-1)$-simplices in $X$).  Then finding $\sigma_n$ or showing it does not
  exist is an integer linear programming problem with the coefficients of these
  chains as variables.

  Now if $\sigma_n$ exists, then it also determines a map $X \to K(\pi_n(Y),n)$.
  Taking the product of these maps for all $n \leq d$ gives us our $\psi_0$.
\end{proof}

We now compute a representative $a_N$ for each coset $N$ of $r_dG \subseteq G$.
Since this is a finite-index subgroup of a fully effective abelian group, this
can be done algorithmically, for example by trying all words of increasing length
in a generating set until a representative of each coset are obtained.  For each
$a_N$, we compute a representative map $\ph_N:X \to H_d$ which is constant on
$A$.  Then the finite set
\[S=\{\psi_N=\psi_0+\YtoH_d\HtoY_d\ph_N:N \in G/r_dG\}\]
contains representatives of the cosets of the action of $[X/A,H_d]$ on
$[X,H_d]^{\YtoH_df}$ obtained by pushing the action on $[X,Y]^f$ forward along
$\YtoH_d$.

Now, for each element of $S$ we apply Theorem \ref*{blocks}\ref{rel-finite} to
the square
\[\xymatrix{
  A \ar[r]^f \ar[d]_\incl & Y_d \ar[d]^{\YtoH_d} \\
  X \ar@{-->}[ru] \ar[r]^{\psi_N} & H_d
}\]
to compute the finite set of preimages under $\YtoH_d$ in $[X,Y_d]^f$.  To obtain
a set of representatives of each coset for the action of $[X/A,H_d]$ on
$[X,Y_d]^f$, we must then eliminate any preimages that are in the same coset.  In
other words, we must check whether two preimages $\tilde\psi$ and $\tilde\psi'$
of $\psi_N$ differ by an element of $[X/A,H_d]$; any such element stabilizes
$\YtoH_d\tilde\psi$, and so its order must divide $r_d$.  Since there are
finitely many elements whose order divides $r_d$, we can check for each such
element $\ph$ in turn whether $[\ph]\cdot[\tilde\psi] \simeq [\tilde\psi']$.

Finally, to finish computing $[X,Y_d]^f$ we must compute the finite stabilizer of
each coset.  This stabilizer is contained in the finite subgroup of $[X/A,H_d]$
of elements whose order divides $r_d$.  Therefore we can again go through all
elements of this subgroup and check whether they stabilize our representative.

\subsection{Summary}
We conclude this section with a formal summary of the algorithm.
\begin{description}
\item[Input]
  \begin{itemize}
  \item A simplicial pair $(X,A)$.
  \item A simplicial complex $Y$, assumed to be simply connected.
  \item A simplicial map $f:A \to Y$.
  \item A positive integer $d$.
  \end{itemize}
\item[Output] If $Y_d$ is not rationally an H-space,
  \textsc{algorithm not applicable}.  Otherwise:
  \begin{itemize}
  \item The $d$th Postnikov stage $Y_d$ of $Y$, represented as a simplicial set
    with effective homology.
  \item A product of Eilenberg--MacLane spaces $H_d$, represented as a simplicial
    set with effective homology.
  \item The group $[X/A,H_d]$, represented as a fully effective abelian group.
  \item A finite (possibly empty) set $\mathcal C$ of maps $\tilde f_i:X \to Y_d$
    representing cosets of the action of $[X/A,H_d]$ on $[X,Y_d]^f$.
  \item For each $i$, the stabilizer of $\tilde f_i$, represented as a finite
    subgroup $\Sigma_i \subseteq [X/A,H_d]$.
  \end{itemize}
\item[Main steps] Here is the outline of the algorithm:
  \begin{enumerate}[A.]
  \item Initialize the computation:
    \begin{itemize}
    \item Compute the homotopy groups of $Y$ through dimension $d$.
    \item Construct the space $H_d=\prod_{n=2}^d K(\pi_n(Y),n)$, and compute the
      group $[X/A,H_d]$.
    \item Set $r_1=1$, and $u_1$, $v_1$, $\act_1$, and $\Hact_1$ to be the unique
      maps between the relevant spaces (which are all points).
    \end{itemize}
  \item
    \textbf{for} $n=2$ through $d$:
    \begin{itemize}
    \item Compute the $k$-invariant $k_n \in H^{n+1}(Y_{n-1};\pi_n(Y))$.  If it is
      of infinite order, \textbf{return} \textsc{algorithm not applicable}.
    \item Otherwise, compute the action of $H_n$ on $Y_n$ and associated data as
      outlined in \S\ref{S:induction}, namely the positive integer $r_n$ and maps
      $u_n$, $v_n$, $\act_n$, and $\Hact_n$.
    \end{itemize}
  \item Using the algorithm of Lemma \ref{lem:extends}, determine whether there
    is a map $X \to H_d$ which extends $v_df:A \to H_d$.
    \begin{itemize}
    \item If there isn't, \textbf{return}
      $(Y_d,H_d,[X/A,H_d],\mathcal C=\emptyset,\emptyset)$.
    \item If there is, compute such a map $\psi_0:X \to H_d$.
    \end{itemize}
  \item \textbf{for} each $N \in [X/A,H_d]/r_d[X/A,H_d]$:
    \begin{itemize}
    \item Choose a representative homotopy class in $[X/A,H_d]$, and a
      representative map $\ph:(X,A) \to (H_d,o)$ in this homotopy class.
    \item Compute the map $\psi=\psi_0+v_du_d\ph:X \to H_d$.
    \item For each homotopy class of maps completing the diagram
      \[\xymatrix{
        A \ar[r]^f \ar[d]_\incl & Y_d \ar[d]^{\YtoH_d} \\
        X \ar@{-->}[ru] \ar[r]^{\psi} & H_d
      }\]
      up to homotopy, compute a representative $\tilde f_i:X \to Y_d$.
    \end{itemize}
    Write $\mathcal C_0$ for the set of all the $\tilde f_i$.
  \item Remove duplicates from $\mathcal C_0$, that is, take a subset
    $\mathcal C \subseteq \mathcal C_0$ which includes only one map from each
    orbit of the action of $[X/A,H_d]$ on $[X,Y_d]^f$.
  \item For each $\tilde f_i \in \mathcal C$, compute the stabilizer as a
    subgroup of the torsion subgroup of $[X/A,H_d]$ and \textbf{return}
    $(Y_d,H_d,[X/A,H_d],\mathcal C,\text{stabilizers})$.
  \end{enumerate}
\end{description}

\section{Variants of Hilbert's tenth problem} \label{S:H10}

In \cite{CKMVW}, the authors show that the existence of an extension is
undecidable by using the undecidability of the existence of solutions to systems
of diophantine equations of particular shapes:
\begin{lem}[Lemma 2.1 of \cite{CKMVW}]
  The solvability in the integers of a system of equations of the form
  \begin{align}
    \sum_{1 \leq i<j \leq r} a_{ij}^{(q)}x_ix_j &= b_q, & q &= 1,\ldots,s
    \label{Q-SYM} \quad \text{or} \tag{Q-SYM} \\
    \sum_{1 \leq i<j \leq r} a_{ij}^{(q)}(x_iy_j-x_jy_i) &= b_q,
    & q &= 1,\ldots,s \label{Q-SKEW} \tag{Q-SKEW}
  \end{align}
  for unknowns $x_i$ and (for \eqrefb{Q-SKEW}) $y_i$, $1 \leq i \leq r$, is
  undecidable.
\end{lem}
We conjecture a very broad generalization of this result.
\begin{conj} \label{conj:Bp}
  For any nonzero bilinear map
  $\mathbf B:\mathbb Z^m \times \mathbb Z^n \to \mathbb Z^p$, the solvability in
  the integers of a system of equations of the form
  \begin{equation}
    \sum_{i,j=1}^r a_{ij}^{(q)}\mathbf B(\mathbf u_i,\mathbf v_j) = \mathbf c_q,
    \qquad q = 1,\ldots,s \label{Q-BIL} \tag{Q-BLIN$(\mathbf B)$} 
  \end{equation}
  for unknowns $\mathbf u_i=(u_{i1},\ldots,u_{im})$ and
  $\mathbf v_j=(v_{j1},\ldots,v_{jn})$, $1 \leq i,j \leq r$, is undecidable.
\end{conj}
We will show this conjecture in certain special cases, particularly the case
$p=1$.  However, the general case would, for instance, imply the undecidability
of Hilbert's tenth problem over the ring of integers of any number field, first
conjectured by Denef and Lipshitz \cite{DL}.  This narrower conjecture is still
open in general, although Mazur and Rubin \cite{MaRu} show using work of Poonen
\cite{Poo2} and Shlapentokh \cite{Shlapen1} that it is implied by the
Shafarevich--Tate conjecture in number theory.  On the other hand,
undecidability is known unconditionally in many cases, for example for totally
real number fields and their quadratic extensions.  For a survey, see
\cite[Theorem 14.1]{Poo1}.

Before discussing the relationship between these two problems, we give a precise
definition:
\begin{defn}
  Given a ring $R$ and a subring $S$, \emph{Hilbert's tenth problem over $R$
    with coefficients in $S$} is the decision problem: given a finite list of
  polynomials in $S[x_1,\ldots,x_n]$, do they have a simultaneous zero in $R^n$?
\end{defn}
\begin{prop}
  Let $R$ be the ring of integers of a number field.  Then Hilbert's tenth
  problem over $R$ with coefficients in $R$ and in $\mathbb Z$ are
  computationally equivalent.
\end{prop}
This is implicit in Poonen's survey \cite{Poo1}; I would like to thank Emil
Je\v r\' abek on MathOverflow for the following proof.
\begin{proof}
  Given a system of polynomials $p_1,\ldots,p_m \in R[x_1,\ldots,x_n]$, we
  construct an equivalent system with coefficients in $\mathbb Z$.  Let
  $\xi \in R$ be such that $R_{(0)}=\mathbb Q(\xi)$.  We introduce a new variable
  $z$ representing $\xi$, and replace the coefficients of each $p_i$ with
  corresponding polynomials in $z$ to obtain polynomials $q_i(x_1,\ldots,x_n,z)$
  with rational coefficients.  Finally we add the minimal polynomial
  $f_\xi(z)$ of $\xi$ to our system.  Then $q_1,\ldots,q_m,f_\xi$ has a solution
  over $R$ if and only if $p_1,\ldots,p_m$ does, since for any $\xi'$ such that
  $f_\xi(\xi')=0$, there is an automorphism of $R$ taking $\xi$ to $\xi'$.

  The polynomials $q_1,\ldots,q_m,f_\xi$ have rational coefficients, and we can
  clear the denominators by multiplying by a sufficiently large integer.
\end{proof}
Now we further reduce the problem to fit in our framework.
\begin{lem} \label{diff}
  Let $R$ be any ring.  If Hilbert's tenth problem with coefficients in a ring
  $S \subseteq R$ is undecidable over $R$, then so is the solvability of a
  system of equations of the form
  \begin{equation}
    \sum_{i,j=1}^r a_{ij}^{(q)}x_iy_j=c_q, \qquad q=1,\ldots,s 
    \label{Q-DIFF} \tag{Q-DIFF}
  \end{equation}
  in unknowns $x_i$ and $y_j$, $1 \leq i,j \leq r$, and again with coefficients
  in $S$.
\end{lem}
\begin{proof}
  The proof exactly follows that of Lemma 2.1 of \cite{CKMVW}, but we give it for
  completeness.  We reduce any system of equations over $R$ to a system of the
  form \eqrefb{Q-DIFF}.  First, we note that any system of equations can be
  converted into a quadratic system by introducing new unknowns representing
  products and powers.  Now to convert a general quadratic system in unknowns
  $z_1,\ldots,z_r$ to a system of the form \eqrefb{Q-DIFF}, we introduce
  variables $x_0,\ldots,x_r$ and $y_0,\ldots,y_r$, replace every quadratic term
  of the form $z_iz_j$ (where $i$ and $j$ are not necessarily distinct) with
  $x_iy_j$, every linear term $z_i$ with $x_iy_0$, and introduce the following
  additional equations:
  \[x_0y_0=1; \qquad x_iy_0-x_0y_i=0, \quad i=1,\ldots,r.\]
  This forces $x_0$ and $y_0$ to be units and inverses of each other; moreover,
  if $x_0,\ldots,x_r,y_0,\ldots,y_r$ is a solution to the newly constructed
  system of the form \eqrefb{Q-DIFF}, then $z_i=x_iy_0=x_0y_i$ is a solution to
  the original quadratic system.  Conversely, given a solution $z_1,\ldots,z_r$
  to the original system, we can take $x_0=y_0=1$ and $x_i=y_i=z_i$ for
  $i=1,\ldots,r$.
\end{proof}
This immediately implies:
\begin{prop}
  Let $R$ be a ring which is finitely generated and free as a $\mathbb Z$-module
  (for example, the ring of integers of a number field, or the matrix ring
  $M_n(\mathbb{Z})$).  Then Hilbert's tenth problem over $R$ with coefficients in
  $\mathbb Z$ is undecidable if and only if \eqrefb{Q-BIL} is undecidable, where
  $\mathbf B$ describes the multiplication law in $R$ in terms of some
  $\mathbb Z$-basis (or, possibly, three different $\mathbb Z$-bases for the left
  factor, the right factor, and the product).
\end{prop}
\begin{ex} \label{ex:Z[i]}
  The solvability in the integers of systems of the form \eqref{Q-BIL} is
  undecidable, where $\mathbf B:\mathbb Z^2 \times \mathbb Z^2 \to \mathbb Z^2$
  is given by
  \[\mathbf B(\mathbf u,\mathbf v)
    =\left(\mathbf u^T \begin{pmatrix}1 & 0 \\ 0 & -1\end{pmatrix}\mathbf v,
    \mathbf u^T \begin{pmatrix}0 & 1 \\ 1 & 0\end{pmatrix}\mathbf v\right).\]
  This bilinear map describes the multiplication law for $\mathbb Z[i]$ in  
  the basis $\{1,i\}$; Hilbert's tenth problem over $\mathbb{Z}[i]$ and any other
  quadratic number ring is undecidable by \cite{Denef}.
\end{ex}

We conclude the section with two more special cases of Conjecture \ref{conj:Bp}.
\begin{prop} \label{rank1}
  Suppose that $\mathbf B:\mathbb Z^m \times \mathbb Z^n \to \mathbb Z^p$ is a
  bilinear map such that for some $L:\mathbb Z^p \to \mathbb Z$,
  $L \circ \mathbf B$ has rank 1.  Then the solvability in the integers of
  systems of the form \eqref{Q-BIL} is undecidable.
\end{prop}
\begin{proof}
  After changes of basis for $\mathbb Z^m$, $\mathbb Z^n$, and $\mathbb Z^p$, we
  can assume that $B_1(\mathbf u,\mathbf v)=cu_1v_1$ for some $c \in \mathbb Z$,
  where $B_1$ is the first coordinate of $\mathbf B$.  Now consider a general
  system of the form \eqref{Q-DIFF}.  We use it to build a corresponding system
  \begin{equation} \label{eqn:rank1}
    \sum_{i,j=1}^r a_{ij}^{(q)}\mathbf B(\mathbf u_i,\mathbf v_j)=
    c_q\mathbf B(\mathbf e_1,\mathbf e_1), \qquad q=1,\ldots,s,
  \end{equation}
  where $\mathbf e_1=(1,0,\ldots,0)$.  We claim this system is equivalent.

  Given a solution $x_1,\ldots,x_r,y_1,\ldots,y_r$ to \eqrefb{Q-DIFF}, clearly
  $x_1\mathbf e_1,\ldots,x_r\mathbf e_r,y_1\mathbf e_1,\ldots,y_r\mathbf e_r$ is a
  solution to \eqrefb{eqn:rank1}.  Conversely, given a solution
  $\mathbf u_1,\ldots \mathbf u_r,\mathbf v_1,\ldots,\mathbf v_r$ to
  \eqrefb{eqn:rank1}, $u_{11},\ldots,u_{r1},v_{11},\ldots,v_{r1}$ is a solution to
  \eqrefb{Q-DIFF}.
\end{proof}
\begin{thm} \label{thm:oneB}
  The solvability in the integers of systems of the form \eqref{Q-BIL} is
  undecidable when $p=1$, that is, when
  $\mathbf B(\mathbf u,\mathbf v)=\mathbf u^TB\mathbf v$ for some $m \times n$
  matrix $B$.
\end{thm}
\begin{rmk} \label{more:p=1}
  One readily sees from the proof that this result admits various
  generalizations:
  \begin{enumerate}
  \item The result holds with the integers replaced by any PID $R$ in which
    Hilbert's tenth problem is undecidable, such as $R=\mathbb{Z}[i]$.  When $R$
    is finite-dimensional and free as a $\mathbb Z$-module, a Diophantine system
    of this form over $R$ with integer coefficients can be reinterpreted as an
    integral Diophantine system of the form (Q-BLIN$(\mathbf A \otimes B)$),
    where $\mathbf A:\mathbb Z^d \otimes \mathbb Z^d \to \mathbb Z^d$ describes
    the multiplication law in $R$ and $\mathbf A \otimes B$ is interpreted as a
    map
    \[(\mathbb Z^d \otimes \mathbb Z^m) \otimes (\mathbb Z^d \otimes \mathbb Z^n)
    \to (\mathbb Z^d \otimes \mathbb Z).\]
    Therefore, the solvability of systems of the form
    (Q-BLIN$(\mathbf A \otimes B)$) is again undecidable.
  \item The result also holds for $p>1$ if the following algebraic condition is
    satisfied: there are decompositions $\mathbb Q^m=L \oplus S$ and
    $\mathbb Q^n=L' \oplus S'$ such that $L$ and $L'$ are one-dimensional and
    the bilinear map $\mathbf B \otimes \mathbb Q:
    \mathbb{Q}^m \otimes \mathbb{Q}^n \to \mathbb{Q}^p$ restricts to zero on
    $L \otimes S'$ and $L' \otimes S$ and is nonzero on $L \otimes L'$.
  \end{enumerate}
\end{rmk}
\begin{rmk}
  Proposition \ref{rank1} and Theorem \ref{thm:oneB} are in some sense opposite
  extremes: the more independent coordinates in the image of $\mathbf B$, the
  likelier one is to find a direction in which the rank is low.  In between we
  have the case where $\mathbf B:\mathbb Z^n \times \mathbb Z^n \to \mathbb Z^n$
  has full rank in every direction; this includes multiplication laws of rings of
  integers of number fields and may be the most difficult situation.
\end{rmk}
\begin{proof}[Proof of Theorem \ref*{thm:oneB}.]
  We show that a system of the form \eqref{Q-DIFF} can be simulated with one of
  the form \ref{Q-BIL}.  By Lemma \ref{diff}, this is sufficient to show that
  solvability of systems of the form \ref{Q-BIL} is undecidable.  The proof is
  again closely related to that of the undecidability of \eqref{Q-SYM} in
  \cite{CKMVW}.

  We first show that we can replace $B$ with a diagonal matrix.
  \begin{lem} \label{lem:fullrank}
    Given an $m \times n$ matrix $B$, there is a square diagonal full-rank
    matrix $B'$ such that for every choice of $\{a_{ij}\}$ and $c_q$, the system
    \begin{equation} \label{eqn:original}
      \sum_{i,j=1}^r a_{ij}^{(q)}\mathbf u_i^TB\mathbf v_j = c_q, \qquad
      q=1,\ldots,s
    \end{equation}
    has a solution if and only if the system
    \begin{equation} \label{eqn:new}
      \sum_{i,j=1}^r a_{ij}^{(q)}(\mathbf u'_i)^TB'\mathbf v'_j = c_q, \qquad q=1,\ldots,s
    \end{equation}
    has a solution.
  \end{lem}
  \begin{proof}
    We can write $B=SAT$ where $A$ is the Smith normal form and $S$ and $T$ are
    invertible $m \times m$ and $n \times n$ matrices, respectively.  Then the
    vectors $(\mathbf u_i,\mathbf v_j)_{i,j=1,\ldots,r}$ are a solution to the
    system \eqrefb{eqn:original} if and only if
    $(S^T\mathbf u_i,T\mathbf v_j)_{i,j=1,\ldots,r}$ are a solution to the system
    \[\sum_{i,j=1}^r a_{ij}^{(q)}\mathbf u_i^TA\mathbf v_j = c_q, \qquad
      q = 1,\ldots,s.\]
    The matrix $A$ consists of a full-rank diagonal submatrix $B'$ in the top
    left corner and zeros everywhere else.  After removing variables which don't
    appear in any terms with nonzero coefficients, we obtain the system
    \eqrefb{eqn:new} with this $B'$.
  \end{proof}
  Thus we may assume that $m=n$ and $B=(b_{k\ell})$ is a diagonal matrix of full
  rank.

  Now consider a general system of the form \eqrefb{Q-DIFF}.  We use it to build
  a system of the form \eqrefb{Q-BIL} with variables
  \begin{align*}
    & u_{i1},\ldots,u_{in}\text{ and }v_{j1},\ldots,v_{jn}, & 1 &\leq i \leq r, \\
    & z_{k\ell}\text{ and }w_{k\ell}, & 1 &\leq k,\ell \leq n.
  \end{align*}
  Define $n \times n$ matrices $Z=(z_{k\ell})$ and $W=(w_{k\ell})$.  Then the
  equations of our new system are
  \begin{equation} \label{inst-BIL}
    \left\{\begin{aligned}
    \sum_{i,j=1}^r a_{ij}^{(q)}\mathbf u_i^TB\mathbf v_j&=b_{11}c_q,
    & q&=1,\ldots,s,\\
    Z^TBW &= B,  \\
    (\mathbf u_i^TBW)_\ell &= 0, & i &= 1,\ldots,r, & \ell &= 2,\ldots,n, \\
    (Z^TB\mathbf v_j)_k &= 0, & j &= 1,\ldots,r, & k &= 2,\ldots,n.
    \end{aligned}\right.
  \end{equation}
  To complete the proof, we must show that the system \eqrefb{inst-BIL} has a
  solution if and only if \eqrefb{Q-DIFF} does.  It is easy to see that
  $\{x_i,y_j\}_{1 \leq i,j \leq r}$ is a solution to \eqrefb{Q-DIFF} if and only if
  \[Z=W=I_n, \qquad \mathbf u_i=x_i\mathbf e_1,\qquad\mathbf v_j=y_j\mathbf e_1,\]
  where $\mathbf e_1$ is the basis vector $(1,0,\ldots,0)$, is a solution to
  \eqrefb{inst-BIL}.  In particular, if \eqrefb{Q-DIFF} has a solution, then so
  does \eqrefb{inst-BIL}.  Conversely, suppose that we have a solution for
  \eqrefb{inst-BIL}.  Since they are integer matrices and $B$ has nonzero
  determinant, $Z$ and $W$ must both have determinant $\pm1$ and are invertible
  over $\mathbb Z$.  Then \eqrefb{inst-BIL} also has the solution
  \[\mathbf u_i'=Z^{-1}\mathbf u_i,\qquad\mathbf v_j'=W^{-1}\mathbf v_j,
  \qquad Z'=W'=I_n,\]
  and $x_i=u_{i1}',y_j=v_{j1}'$ is a solution for \eqrefb{Q-DIFF}.
\end{proof}

\section{Undecidability of extension problems} \label{S:undec}

\begin{thm} \label{thm:undec}
  Let $Y$ be a simply connected finite complex which is not a rational H-space.
  Then the problem of deciding, for a finite simplicial pair $(X,A)$ and a map
  $\ph:A \to Y$, whether an extension to $X$ exists is at least as hard as
  deciding solvability for systems of equations of the form \eqref{Q-BIL}, for a
  bilinear map $\mathbf B$ depending on $Y$.  Moreover, it is enough to
  consider pairs satisfying $\cd(X,A)=d+1$, where $d$ is the smallest degree
  such that $Y_{d}$ is not a rational H-space.
\end{thm}
Examples of target spaces $Y$ for which this gives us a proof of undecidability
include $\mathbb CP^n$ for any $n$, $\mathbb CP^2 \# \mathbb CP^2$, punctured
products of odd-dimensional spheres, Grassmannians, and any $Y$ such that
$\pi_d(Y)$ has rank 1.  In general, one should be able to prove undecidability of
the extension problem for a wide range of target spaces after computing their
Sullivan minimal model.

Before proving the theorem in full generality, we review the proof in
\cite{CKMVW} of the case $Y=S^2$, where undecidability is shown by reduction from
Hilbert's tenth problem for systems of equations of the type \eqref{Q-SYM}.  How
do the authors encode equations in an extension problem?  There are three
ingredients, all encoded into cells of the pair $(X,A)$:
\begin{itemize}
\item Variable cells: copies $S^2_i$ of $S^2$ in $X$ which are not in $A$, and
  hence can be mapped to $Y$ with arbitrary degree.
\item $3$-spheres encoding constant terms of equations: copies $S^3_q$ of $S^3$
  in $A$, which are mapped to $Y$ with a fixed Hopf invariant $b_q$ by the map
  $\ph$.
\item $4$-cells encoding equations.  The $q$th $4$-cell is attached to the rest
  of $X$ by the map
  \[-2b_q\id_{S^3_q}+\sum_{1 \leq i<j \leq r} a_{ij}^{(q)}[\id_{S^2_i},\id_{S^2_j}],\]
  where $[\alpha,\beta]$ denotes the \emph{Whitehead product} of $\alpha$ and
  $\beta$: the composition
  \[S^3 \xrightarrow{\text{attaching map of the top cell of }S^2 \times S^2}
  S^2 \vee S^2 \xrightarrow{\alpha \vee \beta} X^{(3)}.\]
\end{itemize}
In summary, $A$ is a wedge of $3$-spheres and $X$ consists of a wedge of $2$- and
$3$-spheres with $4$-cells attached.

The homotopy class of a map $S^3 \to S^2$ is determined by its Hopf invariant, an
integer.  The Whitehead product $[\id_{S^2},\id_{S^2}]:S^3 \to S^2$ has Hopf
invariant $2$, and the Whitehead product is bilinear in the two variables.
Therefore, the $4$-cells force the degrees $x_i$ on $S^2_i$ of an extension of
$\ph$ to $X$ to satisfy the equations \eqref{Q-SYM}.

The minimal model of $S^2$ is
\[\bigl(\textstyle{\bigwedge}(a^2,b^3),da=0, db=a^2\bigr).\]
The Hopf invariant can be thought of as the result of pairing with $b$.  There is
therefore a relationship between the differential and the Whitehead product:
\[\langle b,[f,g] \rangle=2\langle a,f \rangle\langle a,g \rangle.\]
Such a relationship holds more generally.

In the general case, we use a similar tactic, but with higher-order Whitehead
products, originally defined by Porter \cite{Porter}.  Given spheres
$S^{n_1},\ldots,S^{n_\ord}$, their product can be given a cell structure with one
cell for each subset of $\{1,\ldots,\ord\}$.  Define their \emph{fat wedge}
$\mathbb{V}_{i=1}^\ord S^{n_i}$ to be this cell structure without the top face.
Let $N=-1+\sum_{i=1}^\ord n_i$, and let $\tau:S^N \to \mathbb{V}_{i=1}^\ord S^{n_i}$
be the attaching map of the missing face.  By definition, $\alpha \in \pi_N(Y)$
is contained in the \emph{$\ord$th-order Whitehead product}
$[\alpha_1,\ldots,\alpha_\ord]$, where $\alpha_i \in \pi_{n_i}(Y)$, if it has a
representative which factors through a map
\[S^N \xrightarrow{\tau} \mathbb{V}_{i=1}^\ord S^{n_i} \xrightarrow{f_\alpha} Y\]
such that $[f_\alpha|_{S^{n_i}}]=\alpha_i$.  Note that there are many potential
indeterminacies in how higher-dimensional cells are mapped, so
$[\alpha_1,\ldots,\alpha_\ord]$ is a set of homotopy classes rather than a unique
class.  This set may be empty: for example, if the ordinary Whitehead product
$[\alpha,\beta]$ is nonzero, then $[\alpha,\beta,\gamma]$ is empty for any
$\gamma$ because there is no way to extend the map $\alpha \vee \beta$ to the
product cell.  However, this is not the case in our situation:
\begin{lem} \label{lem:nonempty}
  Suppose that $Y$ is a rational H-space through degree $d-1$.  Then every
  $d$-dimensional higher-order Whitehead product in $Y_{(0)}$ is nonempty.
\end{lem}
\begin{proof}
  Let $\alpha_i:S^{n_i} \to Y_{(0)}$, for $i=1,\ldots,\ord$ be homotopy classes of
  maps, and suppose $\sum_i n_i=d+1$.  Since $Y_{d-1(0)}$ is an H-space,
  $\bigvee_i \alpha_i:\bigvee_i S^{n_i} \to Y_{d-1(0)}$ extends via the H-space
  operation to a map $F:\prod_i S^{n_i} \to Y_{d-1(0)}$.  The obstruction to
  lifting $F$ to $Y_{(0)}$ lies in $(d+1)$-dimensional cohomology, and therefore
  the restriction of $F$ to the fat wedge lifts to $Y_{(0)}$.
\end{proof}

Importantly, higher-order Whitehead products are graded symmetric and multilinear
in a weak sense.  It is easy to see that the factors commute or anticommute as
determined by the grading.  For multilinearity, notice that if two maps
$f,g:S^n \times X \to Y$ agree on $* \times X$, where $* \in S^n$ is some base
point, then there is a well-defined map ``$f+g$'' given by
\begin{equation} \label{pinch}
  S^n \times X \xrightarrow{\text{``pinch the waist''} \times \id}
  (S^n \vee S^n) \times X \xrightarrow{f \cup_{* \times X} g} Y,
\end{equation}
which induces addition in $\pi_n(Y,f(*,x))$ on every fiber $S^n \times \{x\}$.
Likewise, if $f,g:\mathbb{V}_{i=1}^\ord S^{n_i} \to Y$, where $f \circ \tau$ and
$g \circ \tau$ represent elements of $[\alpha_1,\ldots,\alpha_\ord]$ and
$[\alpha_1',\alpha_2,\ldots,\alpha_\ord]$ respectively, agree on
$\prod_{i=2}^\ord S^{n_i}$, then a similar operation yields a well-defined element
$[f+g] \in [\alpha_1+\alpha_1',\alpha_2,\ldots,\alpha_\ord]$.  In particular,
taking $f=g$ and performing the operation repeatedly, we get that
\[[c\alpha_1,\ldots,\alpha_\ord] \supseteq c[\alpha_1,\ldots,\alpha_\ord]\]
(see e.g.~\cite[Theorem 2.13]{Porter}).  Note, however, that there is no more
general notion of additivity.

We can use this weak multilinearity to relate Lemma \ref{lem:nonempty} to $Y$
itself:

\begin{lem} \label{lem:exists}
  For $i=1,\ldots,\ord$, let $\alpha_i \in \pi_{n_i}(Y)$, and denote the image in
  $\pi_{n_i}(Y_{(0)})$ by $\alpha_{i(0)}$.  If
  $[\alpha_{1(0)},\ldots,\alpha_{\ord(0)}] \subseteq \pi_N(Y_{(0)})$ is nonempty,
  then there are positive integers $r_1,\ldots,r_\ord$ such that
  $[r_1\alpha_1,\ldots,r_\ord\alpha_\ord]$ is nonempty in $\pi_N(Y)$.  Moreover,
  if $Y$ is of finite type, then the set of $r_i$ can be chosen to depend only on
  the set of $n_i$.
\end{lem}
\begin{proof}
  We prove a stronger statement: if $\beta \in \pi_N(Y_{(0)})$ is an element of
  $[\alpha_{1(0)},\ldots,\alpha_{\ord(0)}]$, then for some $R=r_1 \cdots r_\ord$,
  $R\beta$ lifts to an element of
  $[r_1\alpha_1,\ldots,r_\ord\alpha_\ord] \subseteq \pi_N(Y)$.

  Let $f:\mathbb{V}_{i=1}^\ord S^{n_i} \to Y_{(0)}$ be a map such that
  $f \circ \tau$ is a representative of $\beta$.  Denote by
  $(r_1,\ldots,r_\ord)f:\mathbb{V}_{i=1}^\ord S^{n_i} \to Y_{(0)}$ the corresponding
  map in which the $i$th factor is multiplied by $r_i$ in the sense of
  \eqref{pinch}.  We construct a map $\mathbb{V}_{i=1}^\ord S^{n_i} \to Y$, one
  cell at a time.  Suppose that $Z$ is a subcomplex of
  $\mathbb{V}_{i=1}^\ord S^{n_i}$ including the boundary of a particular cell $e$
  (WLOG, the top cell of $\prod_{i=1}^s S^{n_i}$ for some $s<\ord$), and
  $F:Z \to Y$ is a lift of $(r_1,\ldots,r_\ord)f$ for some factors
  $r_1,\ldots,r_\ord$.  In particular, if $\tau_s:S^M \to Z$ is the attaching map
  of this cell, then the obstruction to extending $F \circ \tau_s$ to a lift of
  $(r_1,\ldots,r_\ord)f|_e$ is a finite order element of $\pi_{M+1}(Y)$, and there
  is an $r_{\text{new}}>0$ such that $r_{\text{new}}[F \circ \tau_s]=0$.  (In the
  finite type case, we can choose $r_{\text{new}}$ to be the cardinality of the
  torsion subgroup of $\pi_{M+1}(Y)$.)  Using an \eqref{pinch}-type construction
  to multiply the first factor of $F$ by $r_{\text{new}}$, we get a lift of
  $(r_{\text{new}}r_1,r_2,\ldots,r_\ord)f|_{Z \cup e}$ to a map $Z \cup e \to Y$.

  Starting with $S^{n_1} \vee \cdots \vee S^{n_\kappa}$ and performing the operation
  for every higher cell, we obtain a lift of $R\beta$ for some $R$.
\end{proof}

Finally, in the setting of the theorem, certain $d$-dimensional higher-order
Whitehead products don't only exist but are virtually unique:
\begin{lem} \label{lem:uniq}
  Suppose that $Y$ is a rational H-space through degree $d-1$.  Then for some
  $\ord \geq 2$, there is a nonempty $\ord$th-order Whitehead product in
  $\pi_d(Y)$ containing no torsion elements.  Moreover, for the smallest such
  $\ord$, all $\ord$th-order Whitehead products in $\pi_d(Y)$ are unique up to
  torsion.
\end{lem}
\begin{proof}
  Fix a minimal model $\mathcal{M}_Y$ for $Y$ and a basis of generators for the
  indecomposables $V_n$ in each degree $n$ which is dual to a basis for
  $\pi_n(Y)/$torsion.  Since $Y$ is not a rational H-space, there is some least
  $d$ such that the differential in the minimal model $\mathcal{M}_Y$ is
  nontrivial.  Recall that for a minimal model, each nonzero term in the
  differential is at least quadratic.  For each of the generators $\eta$ of
  $V_d$, $d\eta$ is a polynomial in the lower-degree generators.  Denote by
  \emph{P-degree} the degree of an element of the minimal model as a polynomial
  in these generators, as opposed to the degree imposed by the grading.  Let
  $\ord$ be the minimal P-degree of any monomial in any $d\eta$.

  To prove the lemma, we use the connection, first investigated in \cite{AA},
  between the differential in the minimal model and higher-order Whitehead
  products.  The main theorem of \cite{AA}, Theorem 5.4, gives a formula for the
  pairing between an indecomposable $\eta \in V_n$ and any element of an
  $i$th-order Whitehead product set in $\pi_n$, assuming that every term of
  $d\eta$ has P-degree at least $i$.  This formula is somewhat complicated, but
  is $i$-linear in the pairings between factors of the terms of $d\eta$ and
  factors of the Whitehead product.

  It follows that, given an element of a $\ord$th-order Whitehead product set in
  $\pi_d(Y)$, its pairings with each of the generators $\eta$ are given uniquely
  by this formula.  Since $V_d$ is dual to $\pi_d(Y) \otimes \mathbb{Q}$, all
  elements of the $\ord$th-order Whitehead product set are in the same rational
  homotopy class.

  Consider a particular $\eta$ and a particular term $\mu$ of $d\eta$ whose
  P-degree is $\ord$.  Let $\alpha_1,\ldots,\alpha_\ord$ be elements of $\pi_*(Y)$
  dual to the variables in this term.  By Lemmas \ref{lem:nonempty} and
  \ref{lem:exists}, some $[r_1\alpha_1,\ldots,r_\ord\alpha_\ord]$ is nonempty.
  Every element of this set pairs nontrivially with $d\eta$, since there is a
  nonzero contribution from $\mu$ and a zero contribution from all other terms.
  Therefore the set does not contain a torsion element.
\end{proof}

\begin{proof}[Proof of Theorem \ref{thm:undec}]
  We reduce from the problem \eqref{Q-BIL}, for an appropriate bilinear map
  $\mathbf B$.  For each instance of this problem, we construct a pair $(X,A)$
  and map $f:A \to Y$ such that an extension exists if and the instance has a
  solution.

  Fix a minimal model $\mathcal{M}_Y$ for $Y$ and a basis of generators
  $\eta_1,\ldots,\eta_p$ for the indecomposables $V_d$ in degree $d$.  By Lemma
  \ref{lem:uniq}, there is a nontrivial and rationally unique higher-order
  Whitehead product $[\alpha_1,\ldots,\alpha_\ord] \subset \pi_d(Y)$, where
  $\alpha_i \in \pi_{d_i}(Y)$.  Moreover, by Lemma \ref{lem:exists}, we may
  choose $\alpha_3,\ldots,\alpha_\ord$ and positive integers $\rho_1$ and $\rho_2$
  so that for every choice of $\beta \in \rho_1\pi_{d_1}(Y)$ and
  $\gamma \in \rho_2\pi_{d_2}(Y)$, $[\beta,\gamma,\alpha_3,\ldots,\alpha_\ord]$
  is nonempty.

  Now we fix $\alpha_3,\ldots,\alpha_\ord$ and vary $\beta$ and $\gamma$.  For
  each $\eta_k$, the pairing
  $\langle \eta_k,[\beta,\gamma,\alpha_3,\ldots,\alpha_\ord]\rangle$ is
  bilinear in $\beta$ and $\gamma$.  In particular, after fixing
  $\mathbb Z$-bases for $\rho_1\pi_{d_1}(Y) \cong \mathbb Z^m$ and
  $\rho_2\pi_{d_2}(Y) \cong \mathbb Z^n$, we get a bilinear map
  $\mathbf B:\mathbb Z^m \times \mathbb Z^n \to \mathbb Z^p$.

  Now given a system of the form \eqref{Q-BIL}, we will build a
  $(d+1)$-dimensional pair $(X,A)$ and a map $f:A \to Y$ such that the extension
  problem has a solution if and only if the system does.  We define
  \[A=\bigvee_{q=1}^s S^d_q \vee \bigvee_{i=3}^\ord S^{d_i},\]
  and let $f:A \to Y$ send
  \begin{itemize}
  \item $S^{d_i}$ to $Y$ via a representative of $\alpha_i$;
  \item $S^d_q$ to $Y$ via an element of $\pi_d(Y)$ whose pairing with $\eta_k$,
    for each $k$, is $c_{kq}$.
  \end{itemize}
  Finally, we build $X$ from
  $A'=A \vee \bigvee_{i=1}^r S^{d_1}_i \vee \bigvee_{j=1}^r S^{d_2}_j$ as follows:
  \begin{itemize}
  \item Add on cells so that for every $i$ and $j$, $X$ includes the fat wedge
    $\mathbb{V}(S^{d_1}_i,S^{d_2}_j,S^{d_3},\ldots,S^{d_\ord})$, and these fat
    wedges only intersect in $A'$.  Let $\ph_{ij}:S^d \to X$ be the attaching
    map of the missing $(d+1)$-cell for the $(i,j)$th fat wedge.
  \item Add on spheres $S^{d_1\prime}_i$ together with the mapping cylinder of a
    map $S^{d_1}_i \to S^{d_1\prime}_i$ of degree $\rho_1$, and spheres
    $S^{d_2\prime}_j$ together with the mapping cylinder of a map
    $S^{d_2}_j \to S^{d_2\prime}_j$ of degree $\rho_2$.
  \item Then, for each $q$, add a $(d+1)$-cell attached along a representative of
    $\rho([S^d_q]-\sum_{i,j=1}^r a_{ij}^{(q)}[\ph_{ij}])$, where $\rho$ is the
    exponent of the torsion part of $\pi_d(Y)$.
  \end{itemize}
  It is easy to see that $H_n(X,A)=0$ for $n>d$.

  We claim that $(X,A)$ and $f$ pose the desired extension problem.  Indeed, any
  extension of $f$ to $\tilde f:X \to Y$ sends each $S^{d_1}_i$ to an element
  $\beta_i \in \rho_1\pi_{d_1}(Y)$ and each $S^{d_2}_j$ to an element
  $\gamma_j \in \rho_2\pi_{d_2}(Y)$, as constrained by the mapping cylinders.
  Then $\tilde f \circ \ph_{ij}:S^d \to Y$ represents an element of
  $[\beta_i,\gamma_j,\alpha_3,\ldots,\alpha_\ord]$.  Then the $(d+1)$-cells force
  the equations of \eqref{Q-BIL} to hold.

  Conversely, given a satisfying assignment for \eqref{Q-BIL}, there is an
  extension $\tilde f:X \to Y$.  To see this, note that such a satisfying
  assignment gives us values for $\beta_i$ and $\gamma_j$ up to torsion, and by
  construction there is an extension to the fat wedges and the mapping cylinders.
  Moreover, under any such extension, $f_*[S^d_q]$ and
  $\sum_{i,j=1}^r a_{ij}^{(q)}\tilde f_*[\ph_{ij}] \in \pi_d(Y)$ are rationally
  equivalent; thus when multiplied by $\rho$ they are equal, and the map extends
  to the $(d+1)$-cells of $X$.
\end{proof}
\begin{exs}
  We conclude by discussing some instances of $Y$ for which the extension problem
  is undecidable.  For example, if there are $\alpha_3,\ldots,\alpha_\kappa$ such
  that there exists a rationally nontrivial Whitehead product
  $[\alpha_1,\alpha_2,\alpha_3,\ldots,\alpha_\kappa] \in \pi_d(Y)$ and the image
  of the map
  \begin{align*}
    \pi_{d_1}(Y) \otimes \pi_{d_2}(Y) &\to \pi_d(Y) \\
    \beta \otimes \gamma &\mapsto [\beta,\gamma,\alpha_3,\ldots,\alpha_\kappa]
  \end{align*}
  is one dimensional, then the extension problem is undecidable by Theorem
  \ref{thm:oneB}.  This situation includes the cases where $Y$ is a fat wedge of
  odd spheres (in this case the one-dimensional subspace is generated by the
  universal Whitehead product) and $\mathbb CP^n$ (in this case the
  one-dimensional subspace is generated by
  $\underbrace{[\alpha,\ldots,\alpha]}_{n+1\text{ times}}$), as well as any $Y$ such
  that $\pi_d(Y)$ has rank 1.

  A different case is that of $\mathbb CP^2 \# \mathbb CP^2$.  In this case, we
  give explicit generators for the low-dimensional part of the minimal model:
  \[\Bigl(\bigwedge(a_1^2,a_2^2,b^3,c^3,\ldots), da_i=0,
  db=\frac{1}{2}a_1^2-\frac{1}{2}a_2^2, dc=a_1a_2,\ldots\Bigr).\]
  Here the $a_i$ are dual to the two-dimensional classes $\alpha_i$ representing
  the spheres in the two copies of $\mathbb CP^2$.  The $3$-dimensional
  generators are governed by the $4$-cell of $\mathbb CP^2 \# \mathbb CP^2$,
  which ensures that $[\alpha_1,\alpha_1]+[\alpha_2,\alpha_2]=0$.  Then the
  pairing of $b$ and $c$ with Whitehead products of linear combinations of
  $\alpha_1$ and $\alpha_2$ is described by the matrices
  \[\begin{pmatrix}1 & 0 \\ 0 & -1\end{pmatrix} \qquad \text{and} \qquad
    \begin{pmatrix}0 & 1 \\ 1 & 0\end{pmatrix}\]
  from Example \ref{ex:Z[i]}.  Thus the extension problem for maps from
  $4$-complexes to $\mathbb CP^2 \# \mathbb CP^2$ is equivalent to Hilbert's
  tenth problem for $\mathbb Z[i]$, and hence undecidable.

  One can construct similar (though perhaps less natural) examples for other
  number rings.

  Finally, as a demonstration of the wide range of natural examples that can be
  covered with our techniques, we show that the extension problem is undecidable
  for $Y=\widetilde{\Gr}_k(\mathbb R^n)$, the Grassmannian of oriented $k$-planes
  in $\mathbb R^n$, when $2 \leq k \leq n-2$.  The minimal models of these spaces
  are computed explicitly in \cite{MukSan}.  Letting $d$ be the least dimension
  for which the differential is nontrivial, this computation tells us:
  \begin{itemize}
  \item Whenever $k \neq n/2$, or when $k=n/2$ is odd,
    $\pi_d(\widetilde{\Gr}_k(\mathbb R^n))$ has rank $1$, and therefore the
    extension problem for maps into $\widetilde{\Gr}_k(\mathbb R^n)$ is
    undecidable.
  \item When $k=n/2$ and $k$ is even, write $t=k/2$.  In this case, $d=4t-1$,
    $\pi_d(\widetilde{\Gr}_k(\mathbb R^n))$ has rank 2, and the two generators
    have differentials
    \[dv_0=h_t-\tau^2, \qquad du_0=\sigma\tau.\]
    Here $\sigma$ and $\tau$ are generators of degree $k$, and $h_t$ is a
    polynomial in generators $p_1,\ldots,p_{t-1}$, where $p_i$ has degree $4i$,
    and $p_t=\sigma^2$.  Namely, $h_t$ is such that
    \[(1+p_1+\cdots+p_t)(1+h_1+h_2+\cdots)=1.\]
    From this formula one sees that the terms of the differential of multi-degree
    $(k,k)$, which are ``seen'' by pairing with Whitehead products between
    elements of $\pi_k$, are
    \[-\sigma^2-\tau^2,\quad\text{when $t$ is odd;}\qquad
    p_{t/2}^2-\sigma^2-\tau^2,\quad\text{when $t$ is even.}\]
    This lets us write down the bilinear maps
    \[\pi_k(\widetilde{\Gr}_k(\mathbb R^n)) \otimes \mathbb Q \times \pi_k(\widetilde{\Gr}_k(\mathbb R^n)) \otimes \mathbb Q \to \pi_d(\widetilde{\Gr}_k(\mathbb R^n)) \otimes \mathbb Q\]
    induced by the Whitehead product.  When $t$ is odd, the bilinear form given
    by pairing with $2u_0-v_0$ has rank 1, and we can apply Proposition
    \ref{rank1}.  When $t$ is even, we are in the situation of Remark
    \ref{more:p=1}(2).  (Both these claims only rely on the rational structure
    and are independent of integral information.)  In both cases, we know that
    the relevant version of Hilbert's tenth problem is undecidable, and therefore
    so is the extension problem for maps into $\widetilde{\Gr}_k(\mathbb R^n)$.
  \end{itemize}
\end{exs}

\bibliographystyle{amsalpha}
\bibliography{comput}
\end{document}